\documentclass[12pt]{article}

      \usepackage{latexsym}
         \usepackage[reqno, namelimits, sumlimits]{amsmath}
         \usepackage{amssymb, amsfonts}
         \usepackage{amsthm}

 \newtheorem{theorem}{Theorem}[section]
 
 \newtheorem{definition}[theorem]{Definition}

 \newtheorem{remark}[theorem]{Remark}
 
 \newtheorem{cor}[theorem]{Corollary}
 \newtheorem{pro}[theorem]{Proposition}
 
\title{ On global weak solutions to the Cauchy problem for the Navier-Stokes equations with large $L_3$-initial data}

 \author{ G. Seregin
 \footnote{The paper is supported in parts by RFBR project 14-01-00306.}, V. Sverak\footnote{The research was supported in part by grants DMS 1362467 and DMS 1159376 from the National Science Foundation.}
 }
 \newenvironment{dedication}
        {\vspace{6ex}\begin{quotation}\begin{center}\begin{em}}
        {\par\end{em}\end{center}\end{quotation}}

\begin{document}
\maketitle
\abstract{The aim of the note is to discuss different definitions of solutions to the Cauchy problem for the Navier-Stokes equations with the initial data belonging to the Lebesgue space $L_3(\mathbb R^3)$}
\begin{dedication}
\vspace*{0.1cm}
{Dedicated to Professor Nicola Fusco on the occasion of his 60th birthday.}
\end{dedication}\setcounter{equation}{0}

\section{Introduction}


 We consider the classical Cauchy problem for the Navier-Stokes system, describing the flow of a viscous incompressible fluid:
\begin{equation}\label{system}
\left.\begin{array}{rcl}
\partial_tv+v\cdot\nabla v -\Delta v & = & \!\!-\nabla q \\
                          \mbox{div}\,v & = &\,\, \,\,\,\,0
\end{array}\right\}
\qquad \hbox{ in $Q_\infty=\mathbb R^3\times ]0,\infty[$\,\,,}
\end{equation}
with
\begin{equation}\label{bc}
|v(x,t)|\to0\,,\qquad t>0\,,\,|x|\to\infty
\end{equation}
and
\begin{equation}\label{ic}
v(\,\cdot\,,0)=v_0(\,\cdot\,)\in L_3(\mathbb R^3)\end{equation}
where ${\rm div}\, v_0=0$.

There are essentially two methods for constructing the solutions: the perturbation theory and the energy method. In the first approach, we treat the non-linear term as a perturbation and try to find the best spaces in which such treatment is possible. The scaling symmetry of the equation
\begin{equation}\label{vl1}
\begin{array}{c}v(x,t) \\ q(x,t)\end{array}
 \quad\to\quad \begin{array}{c} v_\lambda (x,t)\,=\lambda\,v(\lambda x,\lambda^2 t)\\\,\,\,q_\lambda(x,t) =\lambda^{\!2} q(\lambda x, \lambda^2 t)
 \end{array}
\end{equation}
plays an important r\^ole in the choice of the function spaces, with the scale-invariant spaces  being at the borderline of various families of spaces for which the method works. The most general result in this direction is due to Koch and Tataru~\cite{KochTataru}. The choice of $L_3(\mathbb R^3)$ in~(\ref{ic}) represents a well-known simple example of such a border-line space. The perturbation method cannot work for $L_{3-\delta} (\mathbb R^3)$ for any $\delta >0$.
The perturbation approach goes back to the papers of Oseen and Leray~\cite{Oseen1911, Le}, but in the context of the scale-invariant spaces it was pioneered by Kato~\cite{Kato}.

The energy method is based on the natural {\it  a-priori} energy estimate
\begin{equation}\label{vl2}
\int_{\mathbb R^3} |v(x,t)|^2\,dx +\int_0^t\int_{\mathbb R^3} 2|\nabla v(x,t')|^2\,dx\,dt'\le \int_{\mathbb R^3} |v_0(x)|^2\,dx\,,
\end{equation}
and was pioneered by Leray in~\cite{Le}. The natural condition on the initial data in the context of the energy method is $v_0\in L_2(\mathbb R^3)$. The energy method gives global weak solutions for any initial data in $L_2$, but the regularity and, more importantly, uniqueness of the solutions is unknown, and possibly does not hold, see~\cite{JiaSverak}.

In many cases it is desirable to have a good theory of the weak solution for initial data $v_0\in L_3(\mathbb R^3)$, but the original theory of the weak solutions, which needs $v_0\in L_2(\mathbb R^3)$, does not cover that case. Various approaches have been developed to adapt the theory of the weak solutions so that it would allow $v_0\in L_3(\mathbb R^3)$.
For example, in the paper of Calderon~\cite{Calderon} the author decomposes an $L_3$ initial data $v_0$ as
\begin{equation}\label{vl3}
v_0=v_0^1+v_0^2
\end{equation}
so that $v_0^1$ is small in $L_3$ and $v_0^2$ belongs to $L_2\cap L_3$. Due to the smallness, the initial data $v_0^1$ generates a global smooth solution $v^1$ by perturbation theory, and we can write down the equation for $v^2=v-v^1$ and solve it via the energy method.

A more general approach, to be discussed in some detail below, was developed by Lemarie-Rieusset, see~\cite{LR1}.

Here we consider another method for constructing global weak solutions for $v_0\in L_3(\mathbb R^3)$. The method is very simple and, moreover, is easily extendable to problems in unbounded domains with boundaries. The method of Calderon probably also allows such extensions quite easily, whereas the extension of the (more general) concepts from~\cite{LR1} does not appear to be straighforward.

The main idea is as follows. Let $v^1$ be the solution of the linear version of our problem (obtained from the original system simply by omitting the non-linear term).
We now seek the solution $v$ of the original non-linear problem as
\begin{equation}\label{vl4}
v=v^1+v^2\,.
\end{equation}
It is easy to see that the ``correction" $v^2$ should be in the energy class. The ``first correction" $v^{21}$ is given by
\begin{equation}\label{vl5}
v^{21}_t-\Delta v^{21} + \nabla q^{21}= - {\rm div\,} v^1\otimes v^1\,.
\end{equation}
We have
\begin{equation}\label{vl6}
v^1\otimes v^1 \in L_{\infty}([0,\infty[\,\,,L_{\frac 32}(\mathbb R^3)) \cap  L_{\frac 52}([0,\infty[\,\,,  L_{\frac52}(\mathbb R^3))\subset L_{4}([0,\infty[\,\,,L_2(\mathbb R^3))\,,
\end{equation}
Hence
\begin{equation}\label{vl7}
v^1 \otimes v^1 \in L_2(\mathbb R^3\times (0,T))
\end{equation}
for every $T>0$, which is enough to have $v^{21}$ in the energy class on every bounded time interval. From this it is heuristically clear that we should have
$v=v^1+v^2$, where $v^2$ is in the energy class on every bounded time interval. The general idea that the correction $v^2$ might be easier to deal with than the full solution $v$ is standard, and has been already suggested by considerations in Leray's classical paper~\cite{Le}, and also been often used  in the works on other PDEs.

We now discuss more technical details.
We start with  the definition of the {\it mild solutions} solutions, which is usually considered in connection with the perturbation method for the problem (\ref{system}) -- (\ref{ic}).
\begin{definition}\label{Kato}
A function $u\in C([0,T];L_3(\mathbb R^3))\cap L_5(Q_T)$, satisfying the identity
\begin{equation}\label{integral}
v(x,t)=\int\limits_{\mathbb R^3}\Gamma(x-y,t)v_0(x)dx+\int\limits^t_0\int\limits_{\mathbb R^3}
K(x,y,t-s):v(y,s)\otimes v(y,s)dyds\end{equation}
for all $z=(x,t)\in Q_T:=\mathbb R^3\times]0,T[$, is called a mild solution to problem (\ref{system})-- (\ref{ic}) in $Q_T$. \end{definition}
Here, $\Gamma$ is the known heat kernel and a kernel $K$ is derived with the help of $\Gamma$ as follows:
$$\Delta_y\Phi(x,y,t)=\Gamma(x-y,t), $$$$ K_{mjs}(x,y,t)=\delta_{mj}\frac{\partial^3 \Phi}{\partial y_i\partial y_i\partial y_s}(x,y,t)-\frac{\partial^3 \Phi}{\partial y_m\partial y_j\partial y_s}(x,y,t).$$

Although  mild solutions are known to be unique,
their global existence is an open problem. So,
they exist locally in time and, moreover, in proofs that are available to the authors, the time interval of the existence of mild solutions depends not only on the value $\|v_0\|_{3,\mathbb R^3}$ but on the integral modulus of continuity of $v_0$ in $L_3(\mathbb R^3)$ as well.

The classical existence results about the weak solutions rely upon relatively simple considerations based on the energy estimate~(\ref{vl2}). The main issue in this approach is the uniqueness, which at the moment can only be proved via regularity.  For our set up, such a notion of weak solutions is already known due to Lemarie-Rieusset, see \cite{LR1}. In this paper, we shall call them local energy Leray-Hopf solutions or just local energy solutions or even just Lemarier-Rieusset solutions. The important feature of those solutions is that the very rich $\varepsilon$-regularity theory developed by Caffarelli-Kohn-Nirenberg is applicable to them. Here, it is a definition, which is essentially given by Lemarie-Rieusset, see also \cite{KS}.
\begin{definition}\label{id1}
We call a pair of functions $v$ and $q$ defined in the space-time
cylinder $Q_T=\mathbb R^3\times ]0,T[$ a local energy weak
 Leray-Hopf solution
 to the Cauchy problem (\ref{system})--(\ref{ic}) if
they satisfy the following conditions:
 \begin{eqnarray}\label{i4}
 v\in L_\infty(0,T;L_{2,unif}),\,\nabla v\in L_{2,unif}(0,T), \,
 q\in L_\frac
32(0,T;L_{\frac 32,loc}(\mathbb R^3));&&\end{eqnarray}
\begin{equation} \label{i5}v\,\, and\,\,q\,\, meet\,\, (\ref{system})\,\,in
\,\,the\,\,sense\,\, of\,\,distributions;\end{equation} \begin{equation} \label{i6} the\,
function\,\,t\mapsto\int\limits_{\mathbb R^3}v(x,t)\cdot w(x)\,dx
\,is\,\, continuous\,\,on\,\,[0,T] \end{equation} for any compactly supported
function $w\in L_2(\mathbb R^3)$;

for any compact K, \begin{equation} \label
{i7}\|v(\cdot,t)-v_0(\cdot)\|_{L_{2}(K)}\to 0
\quad as\quad t\to +0;\end{equation}
\begin{eqnarray}\label{i8}
\nonumber \int\limits_{\mathbb R^3}\varphi|v(x,t)|^2\,dx+
2\int\limits_0^{t}\int\limits_{\mathbb R^3}\varphi|\nabla
v|^2\,dxdt\leq \int\limits_0^{t}\int\limits_{\mathbb
R^3}\Big(|v|^2(\partial_t\varphi+\Delta\varphi) &&  \\
 +v\cdot\nabla \varphi(|v|^2+2q)\Big)\,dxdt  &&
\end{eqnarray}
for a.a. $t\in ]0,T[$ and for all nonnegative smooth functions $\varphi$
vanishing in a neighborhood of the parabolic boundary of the
space-time cylinder ${\mathbb R^3}\times ]0,T[$;

 for any
$x_0\in\mathbb R^3$, there exists a function $c_{x_0}\in L_\frac 32
(0,T)$ such that \begin{equation} \label
 {i9}q_{x_0}(x,t):=
q(x,t)-c_{x_0}(t)=q_{x_0}^1(x,t)+q_{x_0}^2(x,t),\end{equation} for $(x,t)\in
B(x_0,3/2)\times ]0,T[$, where
$$q_{x_0}^1(x,t)=-\frac 13 |v(x,t)|^2
+\frac 1{4\pi}\int\limits_{B(x_0,2)}K(x-y): v(y,t)\otimes
v(y,t)\,dy,$$$$q_{x_0}^2(x,t)=\frac 1{4\pi}\int\limits_{\mathbb
R^3\setminus B(x_0,2)}(K(x-y)-K(x_0-y)): v(y,t)\otimes v(y,t)\,dy$$
and $K(x)=\nabla^2 (1/|x|)$.
\end{definition}
Here,  marginal Morrey spaces

$$L_{m,unif}:=\{u\in L_{m,loc}(\mathbb R^3):\,\,\|u\|_{2,unif}=\sup\limits_{x_0\in\mathbb R^3}\|u\|_{m,B(x_0,1)}<\infty\}$$
and
$$L_{m,unif}(0,T):=\{u\in L_m(0,T;L_{m,loc}(\mathbb R^3):\,\, \sup\limits_{x_0\in B(x_0,1)}\int\limits^T_0\int\limits_{B(x_0,1)}|u(x,t)|^mdxdt$$$$<\infty\}$$
have been used.

Lemarie-Rieusset proved  local in time existence of a local energy solution for $v_0\in L_{2,unif}$. But, what seems to be more important, he showed that if $v_0\in E_2$, where
$E_m$ is the completion of $C^\infty_{0,0}(\mathbb R^3):=\{v\in C^\infty_0(\mathbb R^3):\,\,\mbox{div}\, v=0\}$ in $L_m(\mathbb R^3)$, then the above solution exists globally, i.e., for any $T>0$. The corresponding uniqueness theorem is also true saying that if one has two local energy solutions to (\ref{system})--(\ref{ic}) with the same initial data and one of them belongs to $C([0,T];E_3)$, then they coincide on the interval $]0,T[$. It should be noticed that  the space $L_3(\mathbb R^3)$ is continuously imbedded into $E_3$ and of course into $E_2$. So, in this sense, local energy solutions can be regarded as a possible tool
to study the case of initial data belonging to $L_3(\mathbb R^3)$. That has been exploited in the  paper \cite{Ser2012} on the behaviour  of $L_{3}$-norm of a solution  as time tends to a possible blow up. Moreover, as it has been shown there, the limit of a sequence of solutions with weakly converging $L_3$-initial data
is a local energy solution as well. By the way, the same has been proven for initial data from $H^\frac 12 $ in papers, see \cite{RusSve} and \cite{S6}. However, the aforesaid scheme does not work in the case of unbounded domains say as a half space $\mathbb R^3_+$. The reason is simple: it is unknown how to construct local energy solutions in unbounded domains that are different from $\mathbb R^3$. 

The aim of the presented note is to give a 
definition of  global weak solutions to the Cauchy problem for the Navier-Stokes system with $L_3$-initial data that it is not based on the conception of local energy solutions.  This approach seems to be interesting itself and certainly
simplifies the above mentioned proofs in papers \cite{S6} and \cite{Ser2012}. Moreover, it works well for other unbounded domains.

The new  definition relies on two simple facts. Consider a Stokes problem:
\begin{equation}\label{Stokes}
\partial_tv^1-\Delta v^1=-\nabla q^1,\qquad \mbox{div}\,v^1=0
\end{equation}
in $\Omega\times ]0,\infty[$,
\begin{equation}\label{sbc}   v^1(x,t)=0
\end{equation}
for all $(x,t)\in\partial\Omega\times [0,\infty[$, and
\begin{equation}\label{sic}
v^1(\cdot,0)=v_0(\cdot)\in L_3(\Omega).
\end{equation}
Assume that $\Omega\subset\mathbb R^3$ is so good that $v^1$ obeys the following estimates:
\begin{equation}\label{1stStokesest}
\|v^1\|_{3,\infty,\Omega\times]0,\infty[}+\|v^1\|_{5,\Omega\times]0,\infty[}<c\| v_0\|_{3,\Omega}\end{equation}
and
\begin{equation}\label{2ndStokesest}
\|\nabla v^1(\cdot,t)\|_{3,\Omega}\leq \frac c{\sqrt{t}}\| v_0\|_{3,\Omega}\end{equation}
for all $t>0$.
 It is well known that (\ref{1stStokesest}) and (\ref{2ndStokesest}) are satisfied if $\Omega=\mathbb R^3$ or if  $\Omega=\mathbb R^3_+$ (for other cases, see \cite{Giga1986}).

 In what follows, it is assumed that $\Omega=\mathbb R^3$ and thus we may let $q^1=0$. The general case will be discussed elsewhere.

\begin{definition}\label{weakL3sol} Let $v_0\in L_3(\mathbb R^3)$ and let $v^1$ be a solution to problem (\ref{Stokes})--(\ref{sic}). A function $v$, defined in $Q_\infty=\mathbb R^3\times ]0,\infty[$, is called a weak $L_3$-solution to problem (\ref{system})--(\ref{ic}) if $v=v^1+v^2$, where $v^2\in L_{2,\infty}(Q_T)\cap W^{1,0}_2(Q_T)$ with $Q_T=\mathbb R^3\times ]0,T[$ for any $T>0$ and  satisfies the following conditions:
\begin{equation}\label{perturbsystem}
\partial_tv^2+ v^2\cdot\nabla v^2-\Delta v^2+\nabla q^2=-v^1\cdot\nabla v^2 -v^2\cdot\nabla v^1-v^1\cdot\nabla v^1,\quad {\rm div}\,v^2=0\end{equation}
in $Q_\infty$ in the sense of distributions with $q^2\in L_\frac 32(0,T;L_{\frac 32,loc}(\mathbb R^3))$ for $T>0$;

for all $w\in L_2(\mathbb R^3)$, the function
\begin{equation}\label{weakcount}
t\mapsto\int\limits_{\mathbb R^3}v^2(x,t)\cdot w(x)dx
\end{equation}
is continuos at any $t\in [0,\infty[$;
 \begin{equation}\label{incon}
 \|v^2(\cdot,t)\|_{2,\mathbb R^3}\to 0
 \end{equation}
as $t\downarrow  0$;
\begin{equation}\label{eneryineq}
\frac 12\int\limits_{\mathbb R^3}|v^2(x,t)|^2dx+\int\limits^t_0\int\limits_{\mathbb R^3}|\nabla v^2|^2dxds\leq \int\limits^t_0\int\limits_{\mathbb R^3}v^1\otimes v:\nabla v^2
dxds
\end{equation}
for all $t>0$;

for a.a. $t>0 $
$$ \int\limits_{\mathbb R^3}\varphi(x,t)|v^2(x,t)|^2dx+2\int\limits^t_0\int\limits_{\mathbb R^3}\varphi|\nabla v^2|^2dxds\leq $$
 \begin{equation}\label{loceneryineq} \leq \int\limits^t_0\int\limits_{\mathbb R^3}\Big(2v^1\otimes v:\nabla v^2\varphi  +|v^2|^2(\Delta \varphi+\partial_t\varphi)+
\end{equation}
$$+v\cdot\nabla \varphi(|v^2|^2+2q^2+2v^1\cdot v^2)  \Big)dxds
$$
for any non-negative function $\varphi\in C^\infty_0(Q_\infty)$.

\end{definition}
\begin{remark}\label{energybound} If $\|v_0\|_{3,\mathbb R^3}\leq M$, then
\begin{equation}\label{eb}
|v^2|^2_{2,Q_T}:=\|v^2\|^2_{2,\infty, Q_T}+\|\nabla v^2\|^2_{2,Q_T}\leq c(M)\sqrt{T}\end{equation}
\end{remark}
\begin{remark}\label{w=locenergy} Indeed, any weak $L_3$-solution is a local energy solution.
\end{remark}

\begin{theorem}\label{existence}
Problem (\ref{system})--(\ref{ic}) has at least one weak $L_3$-solution.
\end{theorem}

The following important property of weak $L_3$-solutions, in fact, can be regarded as another strong motivation for introducing them.  To explain it, let us consider a sequence $v^{(m)}_0\in L_3(\mathbb R^3)$ such that
$$v^{(m)}_0\rightharpoonup v_0$$
in $L_3(\mathbb R^3)$. We denote by $u^{(m)}$ a weak solution to the Cauchy problem (\ref{system})--(\ref{ic}) with initial data $v^{(m)}_0$.
\begin{theorem}\label{weakconvergence} There exists a subsequence of $u^{(m)}$ (still denoted by $u^{(m)}$) such that:
$$u^{(m)}\rightharpoonup u, \qquad \nabla u^{(m)}\rightharpoonup \nabla u$$
in $L_2(Q_T)$ for any $T>0$
and
$$u^{(m)}\to u$$
in $L_{3{\rm,loc}}(Q_\infty)$, where $u$ is a weak $L_3$-solution to the Cauchy problem (\ref{system})--(\ref{ic}) with initial data $v_0$.
\end{theorem}

Now, let us discuss some uniqueness issues related to weak $L_3$-solutions to the  Cauchy problem (\ref{system})--(\ref{ic}). We start with the following auxiliary statement:
\begin{pro}\label{uniqueness1}
Let $v$ and $\widetilde{v}$ be two weak $L_3$-solutions to the the Cauchy problem for the Navier-Stokes equations corresponding to the initial data $v_0\in L_3(\mathbb R^3)$.
Suppose that $v\in L_{3,\infty, Q_T}$. There exists an absolute constant $\mu>0$ such that if, for some number  $0<T_1\leq T$,
\begin{equation}\label{maincond}
\|v-v_0\|_{3,\infty,Q_{T_1}}\leq \mu,
\end{equation}
then $v=\widetilde v$ in $Q_T$.
\end{pro}

\begin{remark}\label{limitsass} Condition (\ref{maincond}) holds if
\begin{equation}\label{limit}
\lim\limits_{t\downarrow0}\|v(\cdot,t)-v_0(\cdot)\|_{3,\mathbb R^3}=0.
\end{equation}
\end{remark}

An elementary modification of the final part of the proof of Proposition \ref{uniqueness1} gives the following statement.

\begin{theorem}\label{uniqueness2}
Let $v$ and $\widetilde{v}$ be two weak $L_3$-solutions to the the Cauchy problem for the Navier-Stokes equations corresponding to the initial data $v_0\in L_3(\mathbb R^3)$.
Suppose that $v^2\in L_{5} (Q_T).$
Then $v=\widetilde v$ in $Q_T$.
\end{theorem}

\begin{theorem}\label{AppendixtoUni1}
Let $v$  be a weak $L_3$-solutions to the Cauchy problem for the Navier-Stokes equations corresponding to the initial data $v_0\in L_3(\mathbb R^3)
$. Then there exists $T_0=T(v_0)>0$ such that $v\in L_5(Q_{T_0})$. \end{theorem}
\begin{remark}\label{byproduct}
From the proof of Theorem \ref{AppendixtoUni1}, it follows also that $v^2\in L_{3,\infty}(Q_{T_0})$.
\end{remark}

Our final result is (see also \cite{KozSohr} for a different set up):

\begin{theorem}\label{mainuniqueness}
Let $v$ and $\widetilde v$ be weak $L_3$-solutions to the Cauchy problem (\ref{system})--(\ref{ic}) with the same iniyial data $v_0$ from $L_3(\mathbb R^3)$. Let $v\in L_{3,\infty}(Q_T)$. Then $v=\widetilde v$ in $Q_T$.
\end{theorem}

Our paper is rather expository and some statements in it have been already known. We  prove them in order to demonstrate how our new conception of weak $L_3$-solutions works and that it is in a good accordance with the previous definitions of solutions to the Cauchy problem (\ref{system})--(\ref{ic}). We recommend  papers
\cite{Cannone}, \cite{Galdi}, \cite{Giga1986}, \cite{Kato} and monographs \cite{Ladyzhenskaya} and \cite{LR1}
for more details and references.


The paper is organized as follows. In the second section, the exitence of weak $L_3$-solutions is proven. Sequences of weak $L_3$-solutions are studied in the third secion. The uniqueness of weak $L_3$-solutions and related questions are discussed in fourth section. To make the paper more or less self-contained
, we give a simple proof of the existense of mild solutions with the initial data from $L_3(\mathbb R^3)$ in the Appendix.

\setcounter{equation}{0}
\section{Existence}

In this section, we are going to prove Theorem \ref{existence}.

The  first step of our proof  is to solve the problem in bounded domains $\Omega=B(R)$. We do this in a standard way by considering  several simple linear
problems and applying Leray-Schauder principle.

Assume that
\begin{equation}\label{521}
    a\in {\stackrel{\circ}{J}}(\Omega),
\end{equation}
where the space ${\stackrel{\circ}{J}}(\Omega)$ is the completion of $C^\infty_{0,0}(\Omega)$ in $L_2(\Omega)$.
\begin{pro}\label{52p2}
Let $Q_T=\Omega\times ]0,T[$ and
\begin{equation}\label{523}
   w^1, w\in L_\infty(Q_T),\qquad {\rm div}\,w=0\quad{in}\,\,Q_T.
\end{equation}
There exists a unique solution  $u$ to the initial boundary value problem
\begin{eqnarray}\label{524}
\partial_tu-\Delta u+{\rm div}\,u\otimes w+\nabla p=&\nonumber\\=-{\rm div}(w^1\otimes u+u\otimes w^1+w^1\otimes w^1),\quad{\rm div}\,u=0\quad {in}\,\,Q_T,   &\nonumber \\
 u|_{\partial\Omega\times [0,T]}=0,  &\\
 u|_{t=0}=a  &\nonumber
\end{eqnarray}
in the following sense:
$$u\in C([0,T];L_2(\Omega))\cap L_2(0,T;{\stackrel{\circ}{J}}{^1_2}(\Omega)),\quad\partial_t u\in L_2(0,T;({\stackrel{\circ}{J}}{^1_2}(\Omega))');$$
for a.a. $t\in [0,T]$
\begin{eqnarray}\label{525}
   \int\limits_\Omega (\partial_t u(x,t)\cdot \widetilde{v}(x)+\nabla u(x,t): \nabla   \widetilde{v}(x))dx &\nonumber\\
=\int\limits_\Omega ( u(x,t)\otimes w(x,t)+w^1(x,t)\otimes
u(x,t)+&\nonumber\\+u(x,t)\otimes w^1(x,t)+w^1(x,t)\otimes w^1(x,t)): \nabla   \widetilde{v}(x)dx &
\end{eqnarray}
for all $\widetilde{v}\in {\stackrel{\circ}{J}}{^1_2}(\Omega)$;
\begin{equation}\label{526}
    \|u(\cdot,t)-a(\cdot)\|_{2,\Omega}\to 0
\end{equation}
as $t\to +0$.
\end{pro}
Here, we have used the notation ${\stackrel{\circ}{J}}{^1_r}(\Omega)$ for the completion of $C^\infty_{0,0}(\Omega)$ in the Sobolev space $W^1_r(\Omega)$.

\begin{proof} We are going to apply the Leray-Schauder principle. To this end, let
$$X=L_2(0,T;{\stackrel{\circ}{J}}(\Omega)).$$
Given $u\in X$, define $v=A(u)$ as a solution to the following problem:
\begin{equation}\label{527}
 v\in C([0,T];L_2(\Omega))\cap L_2(0,T;{\stackrel{\circ}{J}}{^1_2}(\Omega)),\quad\partial_t v\in L_2(0,T;({\stackrel{\circ}{J}}{^1_2}(\Omega))');
\end{equation}
for a.a. $t\in [0,T]$
\begin{eqnarray}\label{528}
   \int\limits_\Omega (\partial_t v(x,t)\cdot \widetilde{v}(x)+\nabla v(x,t): \nabla   \widetilde{v}(x))dx &\nonumber\\
=\int\limits_\Omega\widetilde{f}(x,t)\cdot
 \widetilde{v}(x)dx &
\end{eqnarray}
for all $\widetilde{v}\in {\stackrel{\circ}{J}}{^1_2}(\Omega)$;
\begin{equation}\label{529}
    \|v(\cdot,t)-a(\cdot)\|_{2,\Omega}\to 0
\end{equation}
as $t\to +0$. Here, $\widetilde{f}=-{\rm div}(u\otimes w+w^1\otimes u+u\otimes w^1+w^1\otimes w^1)$.

Such a function $v$ exists and is unique (for given $u$) 
 since
$$\widetilde{f}\in L_2(0,T;({\stackrel{\circ}{J}}{^1_2}(\Omega))').$$
So, the operator $A$ is well defined. Let us check that it satisfies all the requirements of the Leray-Schauder principle.

\noindent
\underline{Continuity}: Let $v^1=A(u^1)$ and $v^2=A(u^2)$. Then
$$\int\limits_\Omega (\partial_t (v^1-v^2)\cdot \widetilde{v}+\nabla (v^1-v^2): \nabla   \widetilde{v})dx=\int\limits_\Omega \Big((u^1-u^2)\otimes w+w^1\otimes (u^1-u^2)+
$$$$+(u^1-u^2)\otimes w^1\Big):\nabla   \widetilde{v}dx$$
and  letting  $\widetilde{v}=v^1-v^2$, we find
$$\frac 12\partial_t\|v^1-v^2\|^2_{2,\Omega}+\|\nabla v^1-\nabla v^2\|^2_{2,\Omega}\leq c(w,w^1)\|u^1-u^2\|_{2,\Omega}\|\nabla v^1-\nabla v^2\|_{2,\Omega}$$ and thus
$$\sup\limits_{0<t<T}\|v^1-v^2\|_{2,\Omega}\leq c(w,w^1)\|u^1-u^2\|_{2,Q_T}.$$
The latter implies continuity.

\noindent
\underline{Compactness}: As in the previous case, we use the energy estimate
$$\sup\limits_{0<t<T}\|v\|^2_{2,\Omega}+\|\nabla v\|^2_{2,\Omega}\leq c(w,w^1)(\|u\|^2_{2,\Omega}+1).
$$
The second estimate comes from (\ref{528}) and has the form
$$\|\partial_tv\|^2_{L_2(0,T;({\stackrel{\circ}{J}}{^1_2}(\Omega))')}\leq \|\nabla v\|^2_{Q_T}+c(w,w^1)(\|u\|^2_{2,\Omega}+1).
$$
Combining the above bounds, we observe that sets, which are bounded in $X$, remain to be bounded in
$$W=\{w\in L_2(0,T;{\stackrel{\circ}{J}}{^1_2}(\Omega)),\quad \partial_tw\in
L_2(0,T;({\stackrel{\circ}{J}}{^1_2}(\Omega))')\}.$$

Now, 
for $v=\lambda A(v)$ with $\lambda\in [0,1]$, after integration by parts, we find that, for a.a. $t\in [0,T]$, the identity
$$\int\limits_\Omega (\partial_t v\cdot \widetilde{v}+\nabla v: \nabla   \widetilde{v})dx=\lambda\int\limits_\Omega(w^1\otimes v+v\otimes w^1+w^1\otimes w^1):\nabla\widetilde{v}- (w\cdot\nabla v)\cdot \widetilde{v}) dx$$ holds for any $\widetilde{v}\in {\stackrel{\circ}{J}}{^1_2}(\Omega)$.
If we insert $\widetilde{v}(\cdot)=v(\cdot,t)$ into the latter relation, then another identity
$$\int\limits_\Omega(w\cdot\nabla v)\cdot v dx=0$$
ensures the following estimate:
$$\frac 12 \partial_t\int\limits_\Omega|v|^2dx+\int\limits_\Omega|\nabla v|^2dx\leq c(w^1)(\|v\|_{2,\Omega}+1)\|\nabla v\|_{2,\Omega}.$$
Hence,
$$\partial_t\int\limits_\Omega|v|^2dx\leq c(w^1)(\|v\|^2_{2,\Omega}+1)$$
and thus
$$\|v\|^2_{Q_T}\leq c(\|w^1\|_{\infty,\Omega}, T,\|a\|^2_{2,\Omega})=R^2.$$
Now, all the statements of Proposition \ref{52p2} follow from the Leray-Schauder principle. \index{Leray-Schauder principle}
Proposition \ref{52p2} is proved. \end{proof}

Let  $\omega_\varrho$ be a standard mollifier  and let
$$(u)_\varrho(x,t)=\int\limits_\Omega\omega_\varrho(x-x')u(x',t)dx', \qquad(v^1)_\varrho(x,t)=\int\limits_{\mathbb R^3}\omega_\varrho(x-x')v^1(x',t)dx'.$$
It is easy to check  that ${\rm div} (u)_\varrho(\cdot,t)=0$  if $t\mapsto u(\cdot,t)\in {\stackrel{\circ}{J}}(\Omega)$.

Now, we wish to show that, given $\varrho>0$, there exists at least one function $u^\varrho$ such that:
\begin{equation}\label{5210} u^\varrho\in C([0,T];L_2(\Omega))\cap L_2(0,T;{\stackrel{\circ}{J}}{^1_2}(\Omega)),\quad\partial_t u^\varrho\in L_2(0,T;({\stackrel{\circ}{J}}{^1_2}(\Omega))');\end{equation}
for a.a. $t\in [0,T]$
\begin{eqnarray}\label{5211}
   \int\limits_\Omega (\partial_t u^\varrho(x,t)\cdot \widetilde{v}(x)+\nabla u^\varrho(x,t): \nabla   \widetilde{v}(x))dx &\nonumber\\
=\int\limits_\Omega ( u^\varrho(x,t)\otimes (u^\varrho)_\varrho(x,t)+(v^1)_\varrho(x,t)\otimes u^\varrho(x,t)+&\nonumber\\u^\varrho(x,t)\otimes (v^1)_\varrho(x,t)+(v^1)_\varrho(x,t)\otimes (v^1)_\varrho(x,t)): \nabla   \widetilde{v}(x)dx\end{eqnarray}
for all $\widetilde{v}\in {\stackrel{\circ}{J}}{^1_2}(\Omega)$;
\begin{equation}\label{5212}
    \|u^\varrho(\cdot,t)-a(\cdot)\|_{2,\Omega}\to 0
\end{equation}
as $t\to +0$.

We  notice that (\ref{5210})-(\ref{5212}) can be regarded as a weak form of the following initial boundary value problem
\begin{eqnarray}\label{5213}
\partial_tu^\varrho-\Delta u^\varrho+(u^\varrho)_\varrho\cdot \nabla u^\varrho+\nabla p^\varrho=f,\quad{\rm div}\,u^\rho=0\quad {in}\,\,Q_T,   &\nonumber \\
 u^\varrho|_{\partial\Omega\times [0,T]}=0,  &\\
 u^\varrho|_{t=0}=a  &\nonumber
\end{eqnarray}
with $f=-{\rm div}(u^\varrho \otimes(v^1)_\varrho + (v^1)_\varrho \otimes u^\varrho +(v^1)_\varrho\otimes (v^1)_\varrho).$

\begin{pro}\label{52p3} There exists a unique function $u^\varrho$ defined on $Q_\infty
$
 such that  it satisfies (\ref{5210})-(\ref{5212}) for any $T>0$.  
\end{pro}
 \begin{proof} Let us fix an arbitrary $T>0$.

 To simplify our notation, let us  drop upper index $\varrho$ for a moment. The idea is the same as in Proposition \ref{52p2}: to use the Leray-Schauder principle. \index{Leray-Schauder principle} The space $X$ is the same as in the proof of Proposition \ref{52p2}. But the operator $A$ will be defined in a different way: given $u\in X$, we are looking for $w=A(u)$ so that
 \begin{equation}\label{5215} w\in C([0,T];L_2(\Omega))\cap L_2(0,T;{\stackrel{\circ}{J}}{^1_2}(\Omega)),\quad\partial_t w\in L_2(0,T;({\stackrel{\circ}{J}}{^1_2}(\Omega))');\end{equation}
for a.a. $t\in [0,T]$
\begin{eqnarray}\label{5216}
   \int\limits_\Omega (\partial_t w(x,t)\cdot \widetilde{v}(x)+\nabla w(x,t): \nabla   \widetilde{v}(x))dx &\nonumber\\
=\int\limits_\Omega ( w(x,t)\otimes (u)_\varrho(x,t)+(v^1)_\varrho(x,t)\otimes w(x,t)+&\nonumber\\
+w(x,t)\otimes (v^1)_\varrho(x,t)+(v^1)_\varrho(x,t)\otimes (v^1)_\varrho(x,t)): \nabla   \widetilde{v}(x)dx &
\end{eqnarray}
for all $\widetilde{v}\in {\stackrel{\circ}{J}}{^1_2}(\Omega)$;
\begin{equation}\label{5217}
    \|w(\cdot,t)-a(\cdot)\|_{2,\Omega}\to 0
\end{equation}
as $t\to +0$.
By Proposition \ref{52p2}, such a function exists and is unique. 

 \noindent
 \underline{Continuity}: Do the same as in Proposition \ref{52p2}:

 $$ I:=\frac 12 \partial_t\|w^2-w^1\|^2_{2,\Omega}+\|\nabla (w^2-w^1)
 \|^2_{2,\Omega} $$$$=\int\limits_\Omega\Big(w^2\otimes (u^2)_\varrho-w^1\otimes (u^1)_\varrho\Big):\nabla (w^2-w^1)dx $$$$+\int\limits_\Omega\Big((v^1)_\varrho\otimes (w^2-w^1)+(w^2-w^1)\otimes (v^1)_\varrho\Big):\nabla (w^2-w^1)dx$$  $$
= \int\limits_\Omega (w^2-w^1)\otimes (u^2)_\varrho:\nabla (w^2-w^1)dx $$$$ + \int\limits_\Omega   w^1\otimes (u^2-u^1)_\varrho:\nabla (w^2-w^1)dx$$$$+\int\limits_\Omega\Big((v^1)_\varrho\otimes (w^2-w^1)+(w^2-w^1)\otimes (v^1)_\varrho\Big):\nabla (w^2-w^1)dx .$$
The first integral in the right hand side of the above identity vanishes. Hence, 
 $$I\leq 
 \|(u^2-u^1)_\varrho\|_{\infty,\Omega}
 \|w^1\|_{2,\Omega}\|\nabla (w^2-w^1)\|_{2,\Omega}$$$$+c\|(v^1)_\varrho\|_{\infty,\Omega}\| (w^2-w^1)\|_{2,\Omega}\|\nabla (w^2-w^1)\|_{2,\Omega}.$$
 It follows from the H\"older inequality that
 $$\partial_t\|w^2-w^1\|^2_{2,\Omega}+\|\nabla (w^2-w^1)
 \|^2_{2,\Omega}\leq$$$$\leq  c(\varrho)\|w^1\|^2_{2,\Omega}\|u^2-u^1\|^2_{2,\Omega}+c(\varrho)\|v_0\|^2_{3,\mathbb R^3}\| (w^2-w^1)\|^2_{2,\Omega}$$
 and thus
 $$|w^2-w^1|^2_{2,Q_T}\leq c(\varrho,T,\|w^1\|_{2,\infty,Q_T},\|v_0\|_{3,\mathbb R^3})
 \|u^2-u^1\|^2_{2,Q_T}. $$
 The latter gives us continuity.

\noindent
 \underline{Compactness}: In our case, the usual energy estimate implies the following:
\begin{eqnarray}\label{5218}
\frac 12\partial_t\|w\|^2_{2,\Omega}+\|\nabla w\|^2_{2,\Omega}
=- \int\limits_\Omega \Big(w\otimes (u)_\varrho+(v^1)_\varrho\otimes w+&\nonumber\\
+w\otimes (v^1)_\varrho+(v^1)_\varrho\otimes(v^1)_\varrho\Big):\nabla wdx=&\nonumber\\
=- \int\limits_\Omega \Big((v^1)_\varrho\otimes w+w\otimes (v^1)_\varrho+(v^1)_\varrho\otimes(v^1)_\varrho\Big):\nabla wdx&\nonumber\\
\leq c(\|(v^1)_\varrho\|_{5,\Omega}\|w\|_{\frac {10}3,\Omega}+
\|(v^1)_\varrho\|^2_{4,\Omega})\|\nabla w\|_{2,\Omega}\leq\\
\leq c(\|v^1\|_{5,\Omega}\|w\|_{\frac {10}3,\Omega}+
\|v^1\|^2_{4,\Omega})\|\nabla w\|_{2,\Omega}\leq\nonumber\\
\leq c\|v^1\|_{5,\Omega}\|w\|^\frac 25_{2,\Omega}\|\nabla w\|_{2,\Omega}^\frac 85+c\|v^1\|^2_{4,\Omega}\|\nabla w\|_{2,\Omega}\leq \nonumber\\
\leq \frac 12 \|\nabla w\|^2_{2,\Omega}+c\|v^1\|^5_{5,\Omega}\|w\|^2_{2,\Omega}+c\|v^1\|^4_{4,\Omega}.\nonumber\end{eqnarray}
Since
$$\|v^1\|_{4,\Omega} \leq \|v^1\|_{3,\Omega}^\frac 38\|v^1\|_{5,\Omega}^\frac 58,$$
we have
$$\partial_t\|w\|^2_{2,\Omega}+\|\nabla w\|^2_{2,\Omega}\leq c\|v^1\|^5_{5,\Omega}\|w\|^2_{2,\Omega}+c\|v^1\|_{3,\Omega}^\frac 32\|v^1\|_{5,\Omega}^\frac 52.$$
Then
$$\partial_t\Big(\|w\|^2_{2,\Omega}\exp{\Big(-c\int\limits^t_0\|v^1\|_{5,\Omega}^5ds\Big)}\Big)\leq $$$$\leq c\|v^1\|_{3,\Omega}^\frac 32\|v^1\|_{5,\Omega}^\frac 52\exp{\Big(-c\int\limits^t_0\|v^1\|_{5,\Omega}^5ds\Big)}$$
and thus
$$\|w\|^2_{2,\Omega}\leq \|a\|^2_{2,\Omega}\exp{\Big(c\int\limits^t_0\|v^1\|_{5,\Omega}^5ds\Big)}+
$$
$$+c\int\limits^t_0\|v^1(\tau)\|_{3,\Omega}^\frac 32\|v^1(\tau)\|_{5,\Omega}^\frac 52\exp{\Big(c\int\limits^t_\tau\|v^1(s)\|_{5,\Omega}^5ds\Big)}d\tau\leq$$
$$\leq  \|a\|^2_{2,\Omega}\exp{\Big(c\int\limits^t_0\|v^1\|_{5,\Omega}^5ds\Big)}+$$
$$+c\|v^1\|^\frac 32_{3,\infty,Q_T}\int\limits^t_0\|v^1(\tau)\|_{5,\Omega}^\frac 52\exp{\Big(c\int\limits^t_\tau\|v^1(s)\|_{5,\Omega}^5ds\Big)}d\tau.$$
Now, a rough estimate looks like:
$$\|w\|^2_{2,\Omega}\leq \exp{\Big(c\int\limits^t_0\|v^1\|_{5,\Omega}^5ds\Big)}\Big(\|a\|^2_{2,\Omega}+c\|v^1\|^\frac 32_{3,\infty,Q_T}\int\limits^t_0\|v^1(\tau)\|_{5,\Omega}^\frac 52d\tau\Big)\leq$$
$$\leq \exp{\Big(c\int\limits^t_0\|v^1\|_{5,\Omega}^5ds\Big)}\Big(\|a\|^2_{2,\Omega}+c\|v^1\|^\frac 32_{3,\infty,Q_T}\sqrt {t}\|v^1\|_{5,Q_T}^\frac 52\Big).$$
So, we have
$$\|w\|^2_{2,\infty,Q_T}\leq \exp{\Big(c\|v^1\|^5_{5,Q_T}\Big)}
\Big(\|a\|^2_{2,\Omega}+c\|v^1\|^\frac 32_{3,\infty,Q_T}\sqrt {T}\|v^1\|_{5,Q_T}^\frac 52\Big)$$
and
$$\|\nabla w\|^2_{2,Q_T}\leq \|a\|^2_{2,\Omega}+c\|w\|^2_{2,\infty,Q}\|v^1\|^5_{5,Q_T}+c\|v^1\|^\frac 32_{3,\infty,Q_T}\sqrt{T}\|v^1||^\frac 52_{5,Q_T}.$$
Hence, the required energy estimate takes the form
\begin{equation}\label{5219}
|w|^2_{2,Q_T}\leq c(\|v^1\|_{5,Q_T},\|v^1\|_{3,\infty,Q_T},T, \|a\|_{2,\Omega}),\end{equation}
where a constant $c$ is independent of $\varrho$.
\begin{remark}\label{1remark}
If $a=0$, then we can see how constants depends on $T$
\begin{equation}\label{5219a}
|w|^2_{2,Q_T}\leq c(\|v_0\|_{3,\infty})\sqrt{T}.
\end{equation}

\end{remark}

Now, we need to evaluate the first derivative in time. Indeed,
$$
\|\partial_tw\|^2_{({\stackrel{\circ}{J}}{^1_2}(\Omega))'}\leq c\|\nabla w\|^2_{2,\Omega}+
c\int\limits_\Omega\Big(|w|^2|(u)_\varrho|^2+|(v^1)_\varrho |^2|w|^2+|(v^1)_\varrho |^4\Big)dx$$$$
\leq c\|\nabla w\|^2_{2,\Omega}
+c(\varrho)\| w\|^2_{2,\infty,Q_T}\|u\|^2_{2,\Omega}+
c\int\limits_\Omega\Big(|(v^1)_\varrho |^2|w|^2+|(v^1)_\varrho |^4\Big)dx\leq $$
$$\leq c\|\nabla w\|^2_{2,\Omega}
+c(\varrho)\| w\|^2_{2,\infty,Q_T}\|u\|^2_{2,\Omega}+c(\|(v^1)_\varrho\|^2_{5,\Omega}\|w\|^2_{\frac {10}3,\Omega}+
\|(v^1)_\varrho\|^4_{4,\Omega})\leq $$$$\leq c\|\nabla w\|^2_{2,\Omega}
+c(\varrho)\| w\|^2_{2,\infty,Q_T}\|u\|^2_{2,\Omega}+
c(\|v\|^2_{5,\Omega}\|w\|^2_{\frac {10}3,\Omega}+
\|v^1\|^4_{4,\Omega})$$
Then
$$\|\partial_tw\|^2_{L_2(0,T;({\stackrel{\circ}{J}}{^1_2}(\Omega))')}\leq c\|\nabla w\|^2_{2,Q_T}+c(\varrho)\| w\|^2_{2,\infty,Q_T}\|u\|^2_{2,Q_T}+$$
$$+c\|v\|^2_{5,Q_T}\|w\|^2_{\frac {10}3,Q_T}+c\|v^1\|^\frac 32_{3,\infty,Q_T}\sqrt{T}\|v^1||^\frac 52_{5,Q_T}$$
and thus
$$\|\partial_tw\|^2_{L_2(0,T;({\stackrel{\circ}{J}}{^1_2}(\Omega))')}\leq$$
\begin{equation}\label{5220}
\leq c(\varrho,\|v^1\|_{5,Q_T},\|v^1\|_{3,\infty,Q_T},T, \|a\|_{2,\Omega})\Big(1+\|u\|^2_{2,Q_T}\Big).
\end{equation}
 Making use of similar arguments as in the proof of Proposition \ref{52p2}, we conclude that for each fixed $\varrho>0$ the operator $A$ is compact.

Now, 
for $w=\lambda A(w)$ with $\lambda\in [0,1]$, after integration by parts, we find that, for a.a. $t\in [0,T]$,
$$\int\limits_\Omega (\partial_t w\cdot \widetilde{v}+\nabla w: \nabla   \widetilde{v})dx=\lambda\int\limits_\Omega((v^1)_\varrho\otimes w+w\otimes (v^1)_\varrho+$$$$
+(v^1)_\varrho\otimes (v^1)_\varrho):\nabla\widetilde{v}- ((w)_\varrho\cdot\nabla w)\cdot \widetilde{v}) dx$$ for any $\widetilde{v}\in {\stackrel{\circ}{J}}{^1_2}(\Omega)$.
If we insert $\widetilde{v}(\cdot)=w(\cdot,t)$ into the latter relation, then the identity
$$\int\limits_\Omega((w)_\varrho\cdot\nabla w)\cdot w dx=0$$
and previous arguments ensure  the estimate (\ref{5219}).
Now, let us prove the uniqueness for for fixed $\varrho$ and $T$. Coming back to the proof of continuity of the operator $A$, we find
$$\partial_t\|u^2-u^1\|^2_{2,\Omega}+\|\nabla (u^2-u^1)
 \|^2_{2,\Omega}\leq $$$$\leq c(\varrho)\|u^1\|^2_{2,\Omega}\|u^2-u^1\|^2_{2,\Omega}+c(\varrho)\|v_0\|^2_{3,\mathbb R^3}\| (u^2-u^1)\|^2_{2,\Omega}.$$
From this, it follows that $u^1=u^2$ on the interval $]0,T[$. Selecting a a sequence of $T_k\to\infty$ we can construct a unique function $u$ satisfying all statements of the proposition.
Proposition \ref{52p3} is proved. \end{proof}

Now, we wish to extend statements of Proposition \ref{52p3} to
$\Omega=\mathbb R^3$. From now on, let us assume that $a=0$.
\begin{pro}\label{approxwholespace}
Given $\varrho>0$, there exists a unique function $u$, defined on $Q_\infty=\mathbb R^3\times ]0,\infty[$, such that, for any $T>0$, $u\in W^{1,0}_{2,\infty}(Q_T)$ and  $\partial_tu\in L_2(Q_T)$, and $u$ satisfies the idenity
$$\int\limits_{Q_\infty}(-u\cdot\partial_tw+\nabla u:\nabla w)dxdt=$$
\begin{equation}\label{identityvar}
=\int\limits_{Q_\infty}(u\otimes (u)_\varrho+(v^1)_\varrho\otimes u+u\otimes (v^1)_\varrho+(v^1)_\varrho\otimes
(v^1)_\varrho):\nabla wdxdt\end{equation}
for any $w\in C^\infty_{0,0}(Q_\infty)$ and the initial condition $u(\cdot,0)=0$.\end{pro}
 \begin{proof} Let $\varrho$ be fixed and $u^{(k)}$ is a sequence of solutions from Proposition \ref{52p3} for $\Omega_k=B(R_k)$ with $R_k\to\infty$.
 According to (\ref{5219a}), the  energy estimate
\begin{equation}\label{5219k}
|u^{(k)}|^2_{2,Q^k_T}\leq c(\|v_0\|_{3,\mathbb R^3})\sqrt{T}.\end{equation}
holds  for any $T>0$. Here, $Q^k_T=\Omega_k\times ]0,T[$.

Now, let us derive some additional estimates. First, we have
$$\int\limits_{\Omega_k}(|\partial_t u^{(k)}|^2+\frac 12 \partial_t|\nabla  u^{(k)}|^2)dx=-\int\limits_{\Omega_k}
\Big((u^{(k)})_\varrho\cdot \nabla u^{(k)}+(v^1)_\varrho\cdot\nabla u^{(k)}+$$$$+u^{(k)}\cdot \nabla (v^1)_\varrho+
(v^1)_\varrho\cdot \nabla (v^1)_\varrho\Big)\cdot\partial_t
u^{(k)}dx.$$
Moreover, it is easy to check
$$|(u^{(k)})_\varrho|\leq  c(\varrho)\|u^{(k)}\|_{2,\infty,Q^k_T},$$
$$|(v^1)_\varrho|\leq c(\varrho)\|v^1\|_{3,\infty,Q^k_T}\leq c(\varrho)\|v_0\|_{3,\mathbb R^3}.$$
$$|\nabla (v^1)_\varrho|\leq c(\varrho)\|v^1\|_{3,\infty,Q^k_T}\leq c(\varrho) \|v_0\|_{3,\mathbb R^3}.$$
Then, we find
$$\int\limits_{\Omega_k}(|\partial_t u^{(k)}|^2+\partial_t|\nabla  u^{(k)}|^2)dx\leq c(\varrho)\|u^{(k)}\|^2_{2,\infty,Q^k_T}\|\nabla u^{(k)}\|^2_{2,\Omega_k}+ $$$$+c(\varrho)\|v_0\|^2_{3,\mathbb R^3}\|\nabla u^{(k)}\|^2_{2,\Omega_k}+c(\varrho)
\|v_0\|^2_{3,\mathbb R^3}\|u^{(k)}\|^2_{2,\infty,Q^k_T}+$$$$+c(\varrho)\|v_0\|_{3,\mathbb R^3}\int\limits_{\Omega_k}|(v^1)_\varrho|^2|\nabla (v^1)_\varrho|dx\leq$$
$$\leq ...+c(\varrho)\|v_0\|_{3,\mathbb R^3}\|(v^1)_\varrho\|^2_{3,\Omega_k}\|(\nabla v^1)_\varrho\|_{3,\Omega_k}\leq $$
$$\leq ...+c(\varrho)\|v_0\|_{3,\mathbb R^3}\|v^1\|^2_{3,\Omega_k}\|\nabla v^1\|_{3,\Omega_k}\leq $$
$$\leq ...+c(\varrho)\frac 1{\sqrt{t}}\|v_0\|^4_{3,\mathbb R^3} .$$
Integration in $t$ 
 gives
\begin{equation}\label{strong}
\| \partial_t u^{(k)}\|^2_{2,Q^k_T}+\|\nabla  u^{(k)}\|_{2,\infty,Q^k_T}\leq c(\varrho, T, \|v_0\|_{3,\mathbb R^3})
.
\end{equation}

As usual, let us assume that functions $u^{(k)}$ are extended by zero to the whole space $\mathbb R^3$. Then, according
to (\ref{5219k}) and (\ref{strong}), one can select a subsequence (still denoted by $u^{(k)}$) such that
$$u^{(k)}\rightharpoonup u,\qquad \nabla u^{(k)}\rightharpoonup\nabla u, \qquad \partial_t u^{(k)}\rightharpoonup \partial_t u$$
in $L_2(Q_T)$ with $Q_T=\mathbb R^3\times ]0,T[$
and
$$u^{(k)}\to  u$$
for any set of the form $K\times ]0,T[$, where  $K$ is a compact in $ \mathbb R^3$. Moreover, the limit function $u$ satisfies estimates
\begin{equation}\label{energyrho}
|u|^2_{2,Q_T}\leq c(\|v_0\|_{3,\mathbb R^3})\sqrt{T}\end{equation}
and
$$\|\partial_t u\|^2_{2,Q_T}+\|\nabla  u\|_{2,\infty,Q_T}\leq c(\varrho, T, \|v_0\|_{3,\mathbb R^3})
$$
for any $T>0$  
and the identity (\ref{identityvar}).

It remains to prove the uniqueness. We have the same inequality as in the proof of the previous proposition
$$\partial_t\|u^2-u^1\|^2_{2,\mathbb R^3}+\|\nabla (u^2-u^1)
 \|^2_{2,\mathbb R^3}\leq $$$$\leq c(\varrho)\|u^1\|^2_{2,\mathbb R^3}
 \|u^2-u^1\|^2_{2,\mathbb R^3}+c(\varrho)\|v_0\|^2_{3,\mathbb R^3}\| (u^2-u^1)\|^2_{2,\mathbb R^3},$$
which implies the required property. Proposition \ref{approxwholespace} is proven. \end{proof}

Now, we are going to prove the main theorem by passing to the limit  as $\varrho\to0$.
To this end, we shall split $ u$ into four parts
 in the following way:
$$ u=u^{2,1}+u^{2,2}+u^{2,3}+u^{2,4}$$
so that, for $i=1,2,3,4$,
$$\partial_tu^{2,i}-\Delta u^{2,i}+\nabla p^{2,i}= f^i,\qquad \mbox{div}\,u^{2,i}=0$$
in $Q_\infty$,
$$u^{2,i}(x,0)=0$$
for $x\in \mathbb R^3$, where
$$ f^1:=- (u)_\varrho\cdot\nabla u,\quad f^2:=- u\cdot \nabla (v^1)_\varrho, $$$$ f^3:=- (v^1)_\varrho\cdot\nabla u,\quad f^4:=-(v^1)_\varrho\cdot \nabla (v^1)_\varrho.$$
We now can the introduce the pressure $p=p^\varrho$ so that
$$p=p^{2,1}+p^{2,2}+p^{2,3}+p^{2,4}.$$

Let us start with evaluation of $u^{2,1}$. Here, our main tool  is the Solonnikov coercive estimates of the linear theory. 
One can use the standard consequences of the energy bound, the multiplicative inequalities, and H\"older inequality and find
\begin{equation}\label{time22}
\|\partial_tu^{2,1}\|_{s,l,Q_T}+\|\nabla^2u^{2,1}\|_{s,l,Q_T}+\|\nabla p^{2,1}\|_{s,l,Q_T}\leq $$$$\leq c(s)\|f^1\|_{s,l,Q_T}\leq C(s,\|v_0\|_{3,\mathbb R^3})\sqrt{T}
\end{equation}
provided
$$\frac 3s+\frac 2l=4.$$

For $i=2$, we may apply the known estimate of the heat
potential
\begin{equation}\label{standard}
\|\nabla v^1(\cdot,t)\|_{s,\mathbb R^3}\leq \frac c{t^{\frac 1r+\frac 12}}\|v_0(\cdot)\|_{3,\mathbb R^3}\end{equation}
with
$$\frac 1r=\frac 32\Big(\frac 13-\frac 1s\Big).$$
We take $s=4$ and try to estimate $\|f^2\|_{\frac 43,\frac 32,Q_T}$. Indeed,
we have
$$\|f^2(\cdot,t)\|_{\frac 43,\mathbb R^3}\leq \|u(\cdot,t)\cdot\nabla (v^1)_\varrho(\cdot,t)\|_{\frac 43}\leq$$$$\leq \|u(\cdot,t)\|_{2,\mathbb R^3}\|\nabla (v^1)_\varrho(\cdot,t)\|_{4,\mathbb R^3}\leq $$
$$\leq \|u(\cdot,t)\|_{2,\mathbb R^3}\|\nabla v^1(\cdot,t)\|_{4,\mathbb R^3}\leq$$
$$\leq \|u\|_{2,\infty,Q_T}c\frac 1{t^\frac {5}{8}}
\|v_0(\cdot)\|_{3,\mathbb R^3}.$$
Hence,
$$\|f^2\|_{\frac 43,\frac 32,\mathbb R^3}\leq \|u\|_{2,\infty,Q_T}c{T^\frac {1}{24}}
\|v_0(\cdot)\|_{3,\mathbb R^3}\leq c(\|v_0\|_{3,\mathbb R^3})T^\frac 7{24},$$
which implies
$$\|\partial_tu^{2,2}\|_{4/3,3/2,Q_{T}}+\|\nabla^2u^{2,2}\|_{4/3,3/2,Q_T}+\|\nabla p^{2,2}\|_{4/3,3/2,Q_T}\leq $$$$\leq c\|f^2\|_{4/3,3/2,Q_T}\leq C(\|v_0\|_{3,\mathbb R^3})T^\frac 7{24}.
$$

Next, let  $i=3$.  Then, by (\ref{energyrho}),
$$\|f^3(\cdot,t)\|_{6/5,\mathbb R^3}=\|(v^1)_\varrho(\cdot,t)\cdot\nabla u(\cdot,t)\|_{6/5,\mathbb R^3}\leq\|v^1(\cdot,t)\|_{3,\mathbb R^3}\|\nabla u(\cdot,t)\|_{2,\mathbb R^3}\leq $$$$ \leq c(\|v_0\|_{3,\mathbb R^3})\frac 1{t^\frac 14}$$
and thus
$$\|f^3\|_{6/5,3/2,\mathbb R^3}\leq T^\frac 5{12}c(\|v_0\|_{3,\mathbb R^3})
.$$
Applying a Solonnikov estimate one more time, we find
$$\|\partial_tu^{2,3}\|_{6/5,3/2,Q_{T}}+\|\nabla^2u^{2,3}\|_{6/5,3/2,Q_T}+\|\nabla p^{2,3}\|_{6/5,3/2,Q_T}\leq $$$$\leq c\|f^3\|_{6/5,3/2,Q_T}\leq C(\|v_0\|_{3,\mathbb R^3})T^\frac 5{12}.$$
Finally, for the last term, we have
$$\|f^4(\cdot,t)\|_{3/2,\mathbb R^3}:=\|(v^1)_\varrho(\cdot,t)\cdot \nabla (v^1)_\varrho(\cdot,t)\|_{3/2,\mathbb R^3}\leq $$
$$\leq \|v^1(\cdot,t)\|_{3,\mathbb R^3} \|\nabla v^1\|_{3,\mathbb R^3}\leq $$
$$\leq \frac c{t^\frac 12}\|v_0\|^2_{3,\mathbb R^3}$$
and thus
$$\|f^4\|_{3/2,Q_T}\leq cT^\frac 16\|v_0\|^2_{3,\mathbb R^3}.$$
Coercive Solonnikov estimates give
 \begin{equation}\label{timeu24}
\|\partial_tu^{2,4}\|_{3/2,Q_T}+\|\nabla^2u^{2,4}\|_{3/2,Q_T}+\|\nabla p^{2,4}\|_{3/2,Q_T}\leq $$$$\leq c\|f^4\|_{3/2,Q_T}\leq C(\|v_0\|_{3,\mathbb R^3})T^\frac 16.
\end{equation}



In what follows, we are going to use the following Poincare type inequalities:
\begin{equation}\label{pres1}
\int\limits^T_{0}\int\limits_{B(x_0,R)}|p^{2,1}-[p^{2,1}]_{B(x_0,R)}|^\frac 32dx dt\leq c R^\frac 12\int\limits^{T}_{0}\Big(\int\limits_{B(x_0,R)}|\nabla p^{2,1}|^\frac 98dx\Big)^\frac 43dt;\end{equation}
\begin{equation}\label{pres2}\int\limits^T_{0}\int\limits_{B(x_0,R)}|p^{2,2}-[p^{2,2}]_{B(x_0,R)}|^\frac 32dx dt\leq c R^\frac {9}8
\int\limits^T_{0}\Big(\int\limits_{B(x_0,R)}|\nabla p^{2,2}|^\frac 43dx \Big)^\frac 98dt;\end{equation}
\begin{equation}\label{pres3}\int\limits^T_{0}\int\limits_{B(x_0,R)}|p^{2,3}-[p^{2,3}]_{B(x_0,R)}|^\frac 32dx dt\leq c R^\frac 34\int\limits^{T}_{0}\Big(\int\limits_{B(x_0,R)}|\nabla p^{2,3}|^\frac 65dx\Big)^\frac 54dt;\end{equation}
\begin{equation}\label{pres4}
\int\limits^T_{0}\int\limits_{B(x_0,R)}|p^{2,4}-[p^{2,4}]_{B(x_0,R)}|^\frac 32 dx dt\leq cR^\frac 32\int\limits^T_0\int\limits_{B(x_0,R)}|\nabla p^{2,4}|^2dxdt.\end{equation}

Now, let us see what happens if $\varrho\to0$. We can select a subsequence (still denoted as the whole sequence) with the following properties: for any $T>0$,
$$u^\varrho\rightharpoonup u,\qquad \nabla u^\varrho\rightharpoonup\nabla u   $$
in $L_2(Q_T)$,
$$u^\varrho\to u$$
in $L_3(0,T;L_{3,\mbox{loc}}(\mathbb R^3))$,
$$p^\varrho\rightharpoonup p$$
in $L_{\frac 32}(0,T;L_{\frac 32,\mbox{loc}}(\mathbb R^3))$. Moreover, limit functions $u$ and $p$ satisfy the energy estimate
\begin{equation}\label{limitenergy}
|u|_{2,Q_T}\leq c(\|v_0\|_{3,\mathbb R^3})T^\frac 14
\end{equation}
and the Navier-Stokes equations in $Q_\infty$ in the sense of distributions.  From the estimates above, it follows that
the function
$$t\to\int\limits_{\mathbb R^3}u(x,t)\cdot w(x)dx$$
is continuous on $[0,\infty[$ for all $w\in L_2(\mathbb R^3)$.

Let us show that $u$ and $p$ satisfy the local energy inequality. Indeed, we have
$$\int\limits_{\mathbb R^3}\varphi^2(x,t) |u^\varrho(x,t)|^2dx+2\int\limits^t_0\int\limits_{\mathbb R^3}\varphi^2|\nabla u^\varrho|^2dxds\leq $$
$$\leq \int\limits^t_0\int\limits_{\mathbb R^3}\Big(|u^\varrho|^2(\Delta \varphi^2+\partial_t\varphi^2)+(u^\varrho)_\varrho\cdot\nabla \varphi^2(|u^\varrho|^2+2p^\varrho)+$$$$+(v^1)_\varrho\otimes u^\varrho:\nabla u^\varrho\varphi^2+(v^1)_\varrho\otimes u^\varrho:u^\varrho\otimes \nabla\varphi^2+$$
$$+u^\varrho\otimes (v^1)_\varrho:\nabla u^\varrho\varphi^2+u^\varrho\otimes (v^1)_\varrho:u^\varrho\otimes \nabla\varphi^2+$$$$+(v^1)_\varrho\otimes (v^1)_\varrho:\nabla u^\varrho\varphi^2+(v^1)_\varrho\otimes (v^1)_\varrho:u^\varrho\otimes \nabla\varphi^2\Big)dxds$$
for all $\varphi\in C^\infty_0(\mathbb R^4)$.

Further, we observe that
$$\|(u^\varrho)_\varrho-u\|_{3,B(R)\times ]0,T[}=
\|(u^\varrho-u)_\varrho-(u)_\varrho-u\|_{3,B(R)\times ]0,T[}\leq $$
$$\leq \|(u^\varrho-u)_\varrho\|_{3,B(R)\times ]0,T[}+\|(u)_\varrho-u\|_{3,B(R)\times ]0,T[}\leq $$
$$\leq \|(u^\varrho-u)\|_{3,B(2R)\times ]0,T[}+\|(u)_\varrho-u\|_{3,B(R)\times ]0,T[}\to 0$$
as $\varrho\to0$. This, of course, implies
$$\int\limits^t_0\int\limits_{\mathbb R^3}\Big(|u^\varrho|^2(\Delta \varphi^2+\partial_t\varphi^2)+(u^\varrho)_\varrho\cdot\nabla \varphi^2(|u^\varrho|^2+2p^\varrho)\Big)dxds\to$$
$$
\int\limits^t_0\int\limits_{\mathbb R^3}\Big(|u|^2(\Delta \varphi^2+\partial_t\varphi^2)+u\cdot\nabla \varphi^2(|u|^2+2p)\Big)dxds.
$$

Next, first, we notice that
$$\int\limits^t_0\int\limits_{\mathbb R^3}((v^1)_\varrho-v^1)\otimes u^\varrho:\nabla u^\varrho\varphi^2+((v^1)_\varrho-v^1)\otimes u^\varrho:u^\varrho\otimes \nabla \varphi^2\Big)dxds\to0$$
and
$$\int\limits^t_0\int\limits_{\mathbb R^3}\Big(v^1\otimes u^\varrho:\nabla u^\varrho\varphi^2+v^1\otimes u^\varrho:u^\varrho\otimes \nabla \varphi^2\Big)dxds=$$
$$=-\int\limits^t_0\int\limits_{\mathbb R^3}(u^\varrho\cdot\nabla v^1)\cdot u^\varrho\varphi^2dxds.$$
Now, let us consider the case
$$\varphi\in C^\infty_0(\mathbb R^3\times ]0,\infty[).$$
Here, of course, we have
$$-\int\limits^t_0\int\limits_{\mathbb R^3}(u^\varrho\cdot\nabla v^1)\cdot u^\varrho\varphi^2dxds\to-\int\limits^t_0\int\limits_{\mathbb R^3}(u\cdot\nabla v^1)\cdot u\varphi^2dxds=$$
$$=\int\limits^t_0\int\limits_{\mathbb R^3}\Big(v^1\otimes u:\nabla u\varphi^2+v^1\otimes u:u\otimes \nabla \varphi^2\Big)dxds.$$
The same arguments give
$$\int\limits^t_0\int\limits_{\mathbb R^3}\Big(u^\varrho\otimes (v^1)_\varrho:\nabla u^\varrho\varphi^2+u^\varrho\otimes (v^1)_\varrho:u^\varrho\otimes \nabla\varphi^2\Big)dxds\to$$
$$\int\limits^t_0\int\limits_{\mathbb R^3}\Big(u\otimes v^1:\nabla u\varphi^2+u\otimes v^1:u\otimes \nabla\varphi^2\Big)dxds$$
for any $\varphi\in C^\infty_0(\mathbb R^3\times ]0,\infty[)$.
The last term can be treated in the same manner and thus  the following estimate comes out:
$$\int\limits_{\mathbb R^3}\varphi^2(x,t) |u(x,t)|^2dx+2\int\limits^t_0\int\limits_{\mathbb R^3}\varphi^2|\nabla u|^2dxds\leq $$
$$\leq \int\limits^t_0\int\limits_{\mathbb R^3}\Big(|u|^2(\Delta \varphi^2+\partial_t\varphi^2)+u\cdot\nabla \varphi^2(|u|^2+2p)+$$$$+v^1\otimes u:\nabla u\varphi^2+v^1\otimes u:u\otimes \nabla\varphi^2+$$
$$+u\otimes v^1:\nabla u\varphi^2+u\otimes v^1:u\otimes \nabla\varphi^2+$$$$+v^1\otimes v^1:\nabla u\varphi^2+v^1\otimes v^1:u\otimes \nabla\varphi^2\Big)dxds$$
for any $\varphi\in C^\infty_0(\mathbb R^3\times ]0,\infty[)$.

In order to extend our local energy inequality to all function $\varphi\in C^\infty_0(\mathbb R^4)$, we take a function $\chi(t)$ such that $\chi(t)=1$ if $t>\varepsilon>0$ and $\chi(t)=0$ if $t<\varepsilon/2$ with $0\leq \chi'(t)\leq c/\varepsilon$.
Consider $\psi =\chi\varphi$ with $\varphi\in C^\infty_0(\mathbb R^4)$ as a cut-off function in the local energy inequality and see what happens if $\varepsilon\to 0$. The only term we should care of is the term containing the derivative in time:
$$\int\limits^t_0\int\limits_{\mathbb R^3}|u|^2\partial_t\psi^2dxds=\int\limits^t_\varepsilon \int\limits_{\mathbb R^3}\chi^2|u|^2\partial_t\varphi^2dxds+\int\limits^\varepsilon_0\int\limits_{\mathbb R^3}\varphi^2|u|^2\partial_t\chi^2dxds=I_1+I_2.$$
Obviously,
$$I_1\to\int\limits^t_0 \int\limits_{\mathbb R^3}|u|^2\partial_t\varphi^2dxds.$$
To estimates the second term, we can use the energy estimate:
$$|I_2|\leq c|u|^2_{2,\infty,Q_\varepsilon}\leq C(\|v_0\|_{3,\mathbb R^3}) \sqrt{\varepsilon}\to 0.$$
So, the local energy inequality is proven.

From the last inequality, we can deduce that
$$\int\limits_{\mathbb R^3}\varphi^2(x) |u(x,t)|^2dx+2\int\limits^t_0\int\limits_{\mathbb R^3}\varphi^2|\nabla u|^2dxds\leq $$
$$\leq \int\limits^t_0\int\limits_{\mathbb R^3}\Big(|u|^2\Delta \varphi^2+u\cdot\nabla \varphi^2(|u|^2+2p)+$$$$+v^1\otimes u:\nabla u\varphi^2+v^1\otimes u:u\otimes \nabla\varphi^2+$$
$$+u\otimes v^1:\nabla u\varphi^2+u\otimes v^1:u\otimes \nabla\varphi^2+$$$$+v^1\otimes v^1:\nabla u\varphi^2+v^1\otimes v^1:u\otimes \nabla\varphi^2\Big)dxds$$
for any $\varphi\in C^\infty_0(\mathbb R^3)$.

Now, we wish to get a global energy inequality. We can take a function $\varphi$  satisfying $0\leq \varphi\leq 1$ such that $\varphi(x)=1$ if $|x|<R$ and $\varphi(x)=0$ if $|x|>2R$ and $|\nabla \varphi(x)|<c/R$.

The only term to be treated carefully is
$$I= \int\limits^t_0\int\limits_{\mathbb R^3}u\cdot\nabla \varphi^2pdxds.$$

Indeed, we have
$$I=I_1+I_2+I_3+I_4,
$$
where
$$I_i=\int\limits^t_0\int\limits_{\mathbb R^3}u\cdot\nabla \varphi^2(p^{2,i}-[p^{2,i}]_{A(R)})dxds,$$
where $A(R):=B(2R)\setminus \overline{B}(R)$.
Then, by (\ref{pres1}), we have
$$|I_1|\leq \frac cR\|u\|_{3,A(R)\times ]0,t[} \Big(\int\limits^t_0\int\limits_{A(R)}|p^{2,1}-[p^{2,1}]_{A(R)}|^\frac 32dxds\Big)^\frac 23\leq$$
$$\leq \frac cR\|u\|_{3,A(R)\times ]0,t[}  R^\frac 13\Big(\int\limits^{t}_{0}\Big(\int\limits_{A(R)}|\nabla p^{2,1}|^\frac 98dx\Big)^\frac 43dt\Big)^\frac 23\to0$$
as $R\to\infty$.

As to
the second term, we use (\ref{pres2}) and show
$$|I_2|\leq \frac cR\|u\|_{3,A(R)\times ]0,t[} \Big(\int\limits^t_0\int\limits_{A(R)}|p^{2,2}-[p^{2,2}]_{A(R)}|^\frac 32dxds\Big)^\frac 23\leq$$
$$\leq \frac cR\|u\|_{3,A(R)\times ]0,t[}  R^\frac 34\Big(\int\limits^{t}_{0}\Big(\int\limits_{A(R)}|\nabla p^{2,2}|^\frac 43dx\Big)^\frac 98dt\Big)^\frac 23\to0$$
as $R\to\infty$.

From (\ref{pres3}) it follows that
$$|I_3|\leq \frac cR\|u\|_{3,A(R)\times ]0,t[} \Big(\int\limits^t_0\int\limits_{A(R)}|p^{2,3}-[p^{2,3}]_{A(R)}|^\frac 32dxds\Big)^\frac 23\leq$$
$$\leq \frac cR\|u\|_{3,A(R)\times ]0,t[}  R^\frac 12\Big(\int\limits^{t}_{0}\Big(\int\limits_{A(R)}|\nabla p^{2,3}|^\frac 65dx\Big)^\frac 54dt\Big)^\frac 23\to0$$
as $R\to\infty$.

Finally, for the fourth term, we derive from (\ref{pres4})
$$|I_4|\leq \frac cR\|u\|_{3,A(R)\times ]0,t[} \Big(\int\limits^t_0\int\limits_{A(R)}|p^{2,4}-[p^{2,4}]_{A(R)}|^\frac 32dxds\Big)^\frac 23\leq$$
$$\leq \frac cR\|u\|_{3,A(R)\times ]0,t[}  R\Big(\int\limits^{t}_{0}\int\limits_{A(R)}|\nabla p^{2,4}|^\frac 32dxdt\Big)^\frac 23\to0$$
as $R\to\infty$.

The global energy inequality has been proven. That's all.

\setcounter{equation}{0}
\section{Weak Convergence of Initial Data}

Here, we are going to prove Theorem \ref{weakconvergence}.

We know that  $v^{2(m)}$ satisfies the energy estimate
$$|v^{2(m)}|^2_{2,Q_T}\leq c(M)\sqrt{T}$$
for any $T>0$, where
$$M:=\sup\limits_m\|v^{(m)}_0\|_{3,\mathbb R^2}<\infty.$$



Then, one may proceed as in the proof of Theorem \ref{existence},   splitting  $v^{2(m)} $ and $  q^{2(m)} $ so that:
$$  v^{2(m)}=u^{2,1}+u^{2,2}+u^{2,3}+u^{2,4}$$
and
$$q^{2(m)}=p^{2,1}+p^{2,2}+p^{2,3}+p^{2,4},$$
and, for $i=1,2,3,4$,
$$\partial_tu^{2,i}-\Delta u^{2,i}+\nabla p^{2,i}= f^i,\qquad \mbox{div}\,u^{2,i}=0$$
in $Q_\infty$,
$$u^{2,i}(x,0)=0$$
for $x\in \mathbb R^3$, where
$$ f^1:=- v^{2(m)}\cdot\nabla v^{2(m)},\quad f^2:=- v^{2(m)}\cdot \nabla v^{1(m)}, $$$$ f^3:=- v^{1(m)}\cdot\nabla v^{2(m)},\quad f^4:=-v^{1(m)}\cdot \nabla v^{1(m)}.$$

To evaluate  $u^{2,1}$,  Solonnikov's coercive estimates 
and  the energy bound are used. As a result, we find
\begin{equation}\label{2,1}
\|\partial_tu^{2,1}\|_{s,l,Q_T}+\|\nabla^2u^{2,1}\|_{s,l,Q_T}+\|\nabla p^{2,1}\|_{s,l,Q_T}\leq $$$$\leq c(s)\|f^1\|_{s,l,Q_T}\leq C(s,M)\sqrt{T}
\end{equation}
provided
$$\frac 3s+\frac 2l=4.$$

If $i=2$, one can exploit estimates (\ref{standard}) for $v^1$ with $s=4$ and show
$$\|f^2(\cdot,t)\|_{\frac 43,\mathbb R^3}\leq \|v^{2(m)}(\cdot,t)\cdot\nabla v^{1(m)}(\cdot,t)\|_{\frac 43}\leq$$
$$\leq \|v^{2(m)}(\cdot,t)\|_{2,\mathbb R^3}\|\nabla v^{1(m)}(\cdot,t)\|_{4,\mathbb R^3}\leq $$
$$\leq \|v^{2(m)}\|_{2,\infty,Q_T}c\frac 1{t^\frac {5}{8}}
\|v^{(m)}_0(\cdot)\|_{3,\mathbb R^3}.$$
Hence,
$$\|f^2\|_{\frac 43,\frac 32,\mathbb R^3}\leq \|v^{2(m)}\|_{2,\infty,Q_T}c{T^\frac {1}{24}}
\|v^{(m)}_0(\cdot)\|_{3,\mathbb R^3}\leq c(M)T^\frac 7{24},$$
which implies
$$\|\partial_tu^{2,2}\|_{4/3,3/2,Q_{T}}+\|\nabla^2u^{2,2}\|_{4/3,3/2,Q_T}+\|\nabla p^{2,2}\|_{4/3,3/2,Q_T}\leq $$$$\leq c\|f^2\|_{4/3,3/2,Q_T}\leq C(M)T^\frac 7{24}.
$$

If $i=3$,  then similar arguments lead to the bound
$$\|\partial_tu^{2,3}\|_{6/5,3/2,Q_{T}}+\|\nabla^2u^{2,3}\|_{6/5,3/2,Q_T}+\|\nabla p^{2,3}\|_{6/5,3/2,Q_T}\leq $$$$\leq c\|f^3\|_{6/5,3/2,Q_T}\leq C(M)T^\frac 5{12}.$$

Finally, for the last term, we have
 \begin{equation}\label{timeu2,4}
\|\partial_tu^{2,4}\|_{3/2,Q_T}+\|\nabla^2u^{2,4}\|_{3/2,Q_T}+\|\nabla p^{2,4}\|_{3/2,Q_T}\leq $$$$\leq c\|f^4\|_{3/2,Q_T}\leq C(M)T^\frac 16.
\end{equation}

Now, let  $m\to\infty$. We can select a subsequence (still denoted as the whole sequence) such that, for any $T>0$,
$$v^{2(m)}\rightharpoonup v^2,\qquad \nabla v^{2(m)}\rightharpoonup\nabla v^2   $$
in $L_2(Q_T)$,
$$v^{2(m)}\to v^2$$
in $L_3(0,T;L_{3,\mbox{loc}}(\mathbb R^3))$,
$$q^{2(m)}\rightharpoonup q^2$$
in $L_{\frac 32}(0,T;L_{\frac 32,\mbox{loc}}(\mathbb R^3))$. Moreover, limit functions $v=v^1+v^2$ and $q=q^2$ satisfy the estimate
$$|v^2|_{2,Q_T}\leq c(M)T^\frac 14$$
and the Navier-Stokes equations in $Q_\infty$ in the sense of distributions.  It is easy to see that
the function
$$t\to\int\limits_{\mathbb R^3}v^2(x,t)\cdot w(x)dx$$
is continuous on $[0,\infty[$ for all $w\in L_2(\mathbb R^3)$.

Let us show that $v^2$ and $q^2$ satisfy the local energy inequality. Indeed, we have
$$\int\limits_{\mathbb R^3}\varphi^2(x,t) |v^{2(m)}(x,t)|^2dx+2\int\limits^t_0\int\limits_{\mathbb R^3}\varphi^2|\nabla v^{2(m)}|^2dxds $$
$$\leq \int\limits^t_0\int\limits_{\mathbb R^3}\Big(|v^{2(m)}|^2( \varphi^2+\partial_t\varphi^2)+v^{2(m)}\cdot\nabla \varphi^2(|v^{2(m)}|^2+2q^{2(m)})+$$$$+v^{1(m)}\otimes v^{2(m)}:\nabla v^{2(m)}\varphi^2+v^{1(m)}\otimes :v^{2(m)}\otimes \nabla\varphi^2+$$
$$+v^{2(m)}\otimes v^{1(m)}:\nabla v^{2(m)}\varphi^2+v^{2(m)}\otimes v^{1(m)}:v^{2(m)}\otimes \nabla\varphi^2+$$$$+v^{1(m)}\otimes v^{1(m)}:\nabla v^{2(m)}\varphi^2+v^{1(m)}\otimes v^{1(m)} :v^{2(m)}\otimes \nabla\varphi^2\Big)dxds$$
for all $\varphi\in C^\infty_0(\mathbb R^4)$.


The first thing to notice is:
$$\int\limits^t_0\int\limits_{\mathbb R^3}\Big(|v^{2(m)}|^2(\Delta \varphi^2+\partial_t\varphi^2)+v^{2(m)}\cdot\nabla \varphi^2(|v^{2(m)}|^2+2q^{2(m)})\Big)dxds\to$$
$$
\int\limits^t_0\int\limits_{\mathbb R^3}\Big(|v^2|^2(\Delta \varphi^2+\partial_t\varphi^2)+u\cdot\nabla \varphi^2(|v^2|^2+2q^2)\Big)dxds.
$$

As in the previous section, let us consider two cases. In the first one, it is assumed that
$$\varphi\in C^\infty_0(\mathbb R^3\times ]0,\infty[).$$
Then
$$\int\limits^t_0\int\limits_{\mathbb R^3}\Big((v^{1(m)}-v^1)\otimes v^{2(m)}:\nabla v^{2(m)}\varphi^2+$$$$+(v^{1(m)}-v^1)\otimes v^{2(m)}:v^{2(m)}\otimes \nabla \varphi^2\Big)dxds\to0$$
and
$$\int\limits^t_0\int\limits_{\mathbb R^3}\Big(v^1\otimes v^{2(m)}:\nabla v^{2(m)}\varphi^2+v^1\otimes v^{2(m)}:v^{2(m)}\otimes \nabla \varphi^2\Big)dxds=$$
$$=-\int\limits^t_0\int\limits_{\mathbb R^3}(v^{2(m)}\cdot\nabla v^1)\cdot v^{2(m)}\varphi^2dxds.$$

Next, one can observe that
$$-\int\limits^t_0\int\limits_{\mathbb R^3}(v^{2(m)}\cdot\nabla v^1)\cdot v^{2(m)}\varphi^2dxds\to-\int\limits^t_0\int\limits_{\mathbb R^3}(v^2\cdot\nabla v^1)\cdot v^2\varphi^2dxds=$$
$$=\int\limits^t_0\int\limits_{\mathbb R^3}\Big(v^1\otimes v^2:\nabla v^2\varphi^2+v^1\otimes v^2:v^2\otimes \nabla \varphi^2\Big)dxds.$$
The same arguments imply
$$\int\limits^t_0\int\limits_{\mathbb R^3}\Big(v^{2(m)}\otimes v^{1(m)}:\nabla v^{2(m)}\varphi^2+v^{2(m)}\otimes v^{1(m)}:v^{2(m)}\otimes \nabla\varphi^2\Big)dxds\to$$
$$\int\limits^t_0\int\limits_{\mathbb R^3}\Big(v^2\otimes v^1:\nabla v^2\varphi^2+v^2\otimes v^1:v^2\otimes \nabla\varphi^2\Big)dxds$$
for any $\varphi\in C^\infty_0(\mathbb R^3\times ]0,\infty[)$.
The last term is treated in the same manner. Hence,
$$\int\limits_{\mathbb R^3}\varphi^2(x,t) |v^2(x,t)|^2dx+2\int\limits^t_0\int\limits_{\mathbb R^3}\varphi^2|\nabla v^2|^2dxds\leq $$
$$\leq \int\limits^t_0\int\limits_{\mathbb R^3}\Big(|v^2|^2(\Delta \varphi^2+\partial_t\varphi^2)+v^2\cdot\nabla \varphi^2(|v^2|^2+2q^2)+$$$$+v^1\otimes v^2:\nabla v^2\varphi^2+v^1\otimes v^2:v^2\otimes \nabla\varphi^2+$$
$$+v^2\otimes v^1:\nabla v^2\varphi^2+v^2\otimes v^1:v^2\otimes \nabla\varphi^2+$$$$+v^1\otimes v^1:\nabla v^2\varphi^2+v^1\otimes v^1:v^2\otimes \nabla\varphi^2\Big)dxds$$
for any $\varphi\in C^\infty_0(\mathbb R^3\times ]0,\infty[)$. In order to extend our local energy inequality to all function $\varphi\in C^\infty_0(\mathbb R^4)$, we can exploit the same cut-off function $\chi(t)$ as in the proof of the existence theorem.
Letting $\psi =\chi\varphi$ with $\varphi\in C^\infty_0(\mathbb R^4)$, we observe that the only term that should be treated carefully is the term containing the derivative in time. Indeed, for example,
consider the term
$$\int\limits^t_0\chi^2g(s)ds,$$
where
$$g(t)=\int\limits_{\mathbb R^3}v^1\otimes v^2:\nabla v^2\varphi^2dx.$$ We need to show that
$$A:=\int\limits^t_0|g(s)|ds<\infty.$$
Indeed,
$$A\leq \int\limits^t_0\int\limits_{\mathbb R^3}|v^1||v^2||\nabla v^2|dxds\leq \int\limits_0^t\|v^1\|_{5,\mathbb R^3}\|v^2\|_{\frac {10}3,\mathbb R^3}\|\nabla v^2\|_{2,\mathbb R^3}ds\leq$$$$\leq \int\limits_0^t\|v^1\|_{5,\mathbb R^3}\|v^2\|^\frac 25_{2,\mathbb R^3}\|\nabla v^2\|^\frac 85_{2,\mathbb R^3}ds\leq \|v^2\|^\frac 25_{2,\infty,\mathbb R^3}\|v^1\|_{5,Q_t}\|\nabla v^2\|^\frac 85_{2,Q_t}.$$

Now, let us make the evaluation of the most important term
$$\int\limits^t_0\int\limits_{\mathbb R^3}|v^2|^2\partial_t\psi^2dxds=\int\limits^t_\varepsilon \int\limits_{\mathbb R^3}\chi^2|v^2|^2\partial_t\varphi^2dxds+\int\limits^\varepsilon_0\int\limits_{\mathbb R^3}\varphi^2|v^2|^2\partial_t\chi^2dxds=$$$$=I_1+I_2.$$
Obviously,
$$I_1\to\int\limits^t_0 \int\limits_{\mathbb R^3}|v^2|^2\partial_t\varphi^2dxds.$$
To estimates the second term,  the energy estimate is used and, therefore, $$|I_2|\leq c|v^2|^2_{2,\infty,Q_\varepsilon}\leq C(M) \sqrt{\varepsilon}\to 0.$$
So, the local energy inequality has been proven and takes the form:
$$\int\limits_{\mathbb R^3}\varphi^2(x) |v^2(x,t)|^2dx+2\int\limits^t_0\int\limits_{\mathbb R^3}\varphi^2|\nabla v^2|^2dxds\leq $$
$$\leq \int\limits^t_0\int\limits_{\mathbb R^3}\Big(|v^2|^2\Delta \varphi^2+v^2\cdot\nabla \varphi^2(|v^2|^2+2q^2)+$$$$+v^1\otimes v^2:\nabla v^2\varphi^2+v^1\otimes v^2:v^2\otimes \nabla\varphi^2+$$
$$+v^2\otimes v^1:\nabla v^2\varphi^2+v^2\otimes v^1:v^2\otimes \nabla\varphi^2+$$$$+v^1\otimes v^1:\nabla v^2\varphi^2+v^1\otimes v^1:v^2\otimes \nabla\varphi^2\Big)dxds$$
for any $\varphi\in C^\infty_0(\mathbb R^3)$.

From the latter relation we can deduce the global energy inequality, using the same arguments as in the proof of the existence theorem.
Hence, the limit function $v=v^1+v^2$ is  a weak $L_3$-solution starting with initial data $v_0$.

\setcounter{equation}{0}
\section{Uniqueness}
Let us start with a proof of Proposition \ref{uniqueness1}.


\begin{proof}   Our first remark is that, given $\varepsilon>0$ and $R>0$, 
there exists a number
$R_*(T,R,\varepsilon)  >0$ such that
if $B(x_0,R)\subset \mathbb R^3\setminus B(R_*)$ and $t_0-R^2>0$ then
$$\frac 1{R^2}\int\limits_{Q(z_0,R)}(|v|^3+|q-[q]_{B(x_0,R)}|^\frac 32)dx dt \leq \varepsilon .$$
For $v$ and $q^1=0$, it is certainly true. For $q^2$, we can use arguments similar to those used in the previous section. Indeed, if $q^2=p^{2,1}+p^{2,2}+p^{2,3}+p^{2,4}$, then, for example, we have
$$\frac 1{R^2}\int\limits_{Q(z_0,R)}|p^{2,1}-[p^{2,1}]_{B(x_0,R)}|^\frac 32dxds\leq
$$
$$ \leq\frac 1{R^2}\int\limits^T_0\int\limits_{B(x_0,R)}|p^{2,1}-[p^{2,1}]_{B(x_0,R)}|^\frac 32dxds\leq \frac 1{R^\frac 32}\int\limits^{T}_{0}\Big(\int\limits_{B(x_0,R))}|\nabla p^{2,1}|^\frac 98dx\Big)^\frac 43dt\leq $$$$\leq \frac 1{R^\frac 32}\int\limits^{T}_{0}\Big(\int\limits_{\mathbb R^3\setminus B(R_*))}|\nabla p^{2,1}|^\frac 98dx\Big)^\frac 43dt\to0$$
as $R_*\to\infty$ for any fixed $R>0$. Since $v\in L_{3,\infty}(Q_T)$, it is not so difficult to show that the pair $v$ and $q^2$ satisfies the local energy inequality (in fact, the local energy identity) and thus,
by $\varepsilon$-regularity theory developed    in \cite{CKN}, we can claim that
$$|v(z_0)|\leq \frac cR$$
as long as $z_0$ and $R$ satisfy the conditions above. According to \cite{ESS2003}, $v$ is locally bounded as it belongs to $L_{3,\infty}(Q_T)$. Therefore, we can ensure that $v\in L_\infty(Q_{\delta,T})$ for any $\delta>0$. Here, $Q_{\delta,T}=\mathbb R^3\times ]\delta,T[$. Then, we can easily show that, for any $\delta>0$,
$v^2\in W^{2,1}_{2}(Q_{\delta,T})$, $\nabla v^2\in L_{2,\infty}(Q_{\delta,T})$, and $\nabla q^2\in L_2(Q_{\delta,T})$. The latter allows us to state that   the energy identity
$$\frac 12\int\limits_{\mathbb R^3}|v^2(x,t)|^2dx+\int\limits^t_0\int\limits_{\mathbb R^3}|\nabla v^2|^2dxds= \int\limits^t_0\int\limits_{\mathbb R^3}v^1\otimes v:\nabla v^2
dxds$$
holds for any $t>0$ and, moreover,
$$ \int\limits_{\mathbb R^3}\Big(\partial_tv^2(x,t)\cdot w(x)+(v(x,t)\cdot\nabla v(x,t))\cdot w(x)+\nabla v^2(x,t):\nabla w(x)\Big)dx=0 $$
for any $w\in C^\infty_{0,0}(\mathbb R^3)$ and for a.a. $t\in ]0,t[$.

As $\widetilde v^2$, we have
$$\int\limits_{Q_T}\Big(-\widetilde v^2\cdot \partial_tw-\widetilde v\otimes \widetilde v:\nabla w+\nabla \widetilde v^2:\nabla w\Big)dxdt=0$$
for any $w\in C^\infty_{0,0}(Q_T)$.
 Using known approximative arguments, see for example \cite{mybook}, we deduce from the last identity that
 $$\int\limits_\delta^{t_0}\Big(-\widetilde v^2\partial_t v^2-\widetilde v\otimes \widetilde v:\nabla v^2+\nabla \widetilde v^2:\nabla v^2\Big)dxdt+$$
$$+\int\limits_{\mathbb R^3}\widetilde v^2(x,t_0)\cdot v^2(x,t_0)dx-\int\limits_{\mathbb R^3}\widetilde v^2(x,\delta)\cdot v^2(x,\delta)dx=0$$
for any $t_0\in [\delta,T]$.

Let $w=\widetilde v^2-v^2$. Then
$$\int\limits_\delta^{t_0}\int\limits_{\mathbb R^3}\partial_tv^2\cdot \widetilde v^2dxdt-\frac 12\int\limits_{\mathbb R^3}|v^2(x,t_0)|^2dx+\frac 12\int\limits_{\mathbb R^3}|v^2(x,\delta)|^2dx+$$
$$+\int\limits_\delta^{t_0}\int\limits_{\mathbb R^3}\Big(-v\otimes v:\nabla w+\nabla v^2:\nabla w\Big)dxdt=0.$$

Adding the last two identities, we find
$$\int\limits_{\mathbb R^3}\Big[\widetilde v^2(x,t_0)\cdot v^2(x,t_0)-\widetilde v^2(x,\delta)\cdot v^2(x,\delta)-\frac 12|v^2(x,t_0)|^2+\frac 12|v^2(x,\delta)|^2\Big]dx+$$
$$+\int\limits_\delta^{t_0}\int\limits_{\mathbb R^3}\Big(-v\otimes v:\nabla w-\widetilde v\otimes \widetilde v:\nabla v^2+\nabla v^2:\nabla w+\nabla \widetilde v^2:\nabla v^2\Big)dxdt=0$$
or
$$\int\limits_{\mathbb R^3}\Big[\widetilde v^2(x,t_0)\cdot v^2(x,t_0)-\frac 12|v^2(x,t_0)|^2\Big]dx+$$
$$+\int\limits_0^{t_0}\int\limits_{\mathbb R^3}\Big(-v\otimes v:\nabla w-\widetilde v\otimes \widetilde v:\nabla v^2+\nabla v^2:\nabla w+\nabla \widetilde v^2:\nabla v^2\Big)dxdt=\alpha(\delta),$$
where $\alpha(\delta)\to 0$ as $\delta\to0$.

We also know that   $\widetilde v^2$ satisfies the energy inequality
$$\frac 12 \int\limits_{\mathbb R^3}|\widetilde v^2(x,t_0)|^2dx+\int\limits_0^{t_0}\int\limits_{\mathbb R^3}|\nabla \widetilde v^2|^2dxdt\leq$$
$$\leq  \int\limits_0^{t_0}\int\limits_{\mathbb R^3}\widetilde v\otimes \widetilde v:\nabla \widetilde v^2dxdt.$$
Subtracting the previous identity from the last inequality, we show
$$\frac 12  \int\limits_{\mathbb R^3}|w(x,t_0)|^2dx+\int\limits_0^{t_0}\int\limits_{\mathbb R^3}|\nabla \widetilde v^2|^2dxdt\leq $$$$\leq  \int\limits_0^{t_0}\int\limits_{\mathbb R^3}\widetilde v\otimes \widetilde v:\nabla \widetilde v^2dxdt+$$
$$+\int\limits_0^{t_0}\int\limits_{\mathbb R^3}\Big(-v\otimes v:\nabla w-\widetilde v\otimes \widetilde v:\nabla v^2+\nabla v^2:\nabla w+\nabla \widetilde v^2:\nabla v^2\Big)dxdt-\alpha(\delta)$$
and then, passing to the limit as $\delta\to0$, we find
$$\frac 12  \int\limits_{\mathbb R^3}|w(x,t_0)|^2dx+\int\limits_0^{t_0}\int\limits_{\mathbb R^3}|\nabla w|^2dxdt\leq $$
$$\leq \int\limits_0^{t_0}\int\limits_{\mathbb R^3}\Big(\widetilde v\otimes \widetilde v:\nabla w- v\otimes v:\nabla w\Big)dxdt=\int\limits_0^{t_0}\int\limits_{\mathbb R^3}(w\otimes v+v\otimes w):\nabla wdxdt.$$
So, finally,
$$\int\limits_{\mathbb R^3}|w(x,t_0)|^2dx+\int\limits_0^{t_0}\int\limits_{\mathbb R^3}|\nabla w|^2dxdt\leq$$$$\leq  c\int\limits_0^{t_0}\int\limits_{\mathbb R^3}|v^1|^2|w|^2dxdt+c\int\limits_0^{t_0}\int\limits_{\mathbb R^3}|v^2|^2|w|^2dxdt=cI_1+cI_2.$$

Estimate for $I_1$ has been already derived:
$$I_1\leq c\Big(\int\limits_0^{t_0}\int\limits_{\mathbb R^3}|v^1(y,t)|^5dy\int\limits_{\mathbb R^3}|w(x,t)|^2dxdt\Big)^\frac 25
\Big(\int\limits_0^{t_0}\int\limits_{\mathbb R^3}|\nabla w|^2dxdt\Big)^\frac 35\leq $$$$\leq c\Big(\int\limits_0^{t_0}g^1(t)\int\limits_{\mathbb R^3}|w(x,t)|^2dxdt\Big)^\frac 25
\Big(\int\limits_0^{t_0}\int\limits_{\mathbb R^3}|\nabla w|^2dxdt\Big)^\frac 35 , $$
where
$$g^1(t):=\int\limits_{\mathbb R^3}|v^1(y,t)|^5dy.$$

To estimate the second term, we are going to exploit condition (\ref{maincond}) that implies the existence of $\delta\in ]0, T_1]$ with the following property:
$$\|v^2\|_{3,\infty, Q_{\delta}}\leq 2\mu.$$
Then for $t_0\leq \delta$, we have
$$I_2\leq \int\limits_0^{t_0}\|v^2\|^2_{3,\mathbb R^3}\|w\|^2_{6,\mathbb R^3}dt\leq$$
$$\leq c\|v^2\|^2_{3,\infty,Q_{t_0}}\int\limits_0^{t_0}\int\limits_{\mathbb R^3}|\nabla w|^2dxdt\leq 4\mu^2c\int\limits_0^{t_0}\int\limits_{\mathbb R^3}|\nabla w|^2dxdt.$$
Letting
$$8\mu^2c=1,$$
 we find
$$\int\limits_{\mathbb R^3}|w(x,t_0)|^2dx+\int\limits_0^{t_0}\int\limits_{\mathbb R^3}|\nabla w|^2dxdt\leq c\int\limits_0^{t_0}g^1(t)\int\limits_{\mathbb R^3}|w(x,t)|^2dxdt$$
for $0<t_0\leq\delta.$
Hence, $w=0$ on the interval $]0,\delta[$. On the other hand, we know that $v\in L_5(Q_{\delta,T}) $ and thus $v^2\in L_5(Q_{\delta,T}) $ and the same arguments as above give $w=0$ on the whole interval $]0,T[$.
\end{proof}

Now, we wish to prove Theorem \ref{AppendixtoUni1}.
\begin{proof}
It is enough to show that there exists a weak $L_3$-solution $u$ that belongs to $L_5(Q_{T_0})$  for some $T_0>0$ with the same initial data. Indeed, by Theorem \ref{uniqueness2}, $v=u$ in $Q_{T_0}$. To this end, let us go back to our approximating solution $v^\varrho$. For this smooth solution, we have the known estimate
$$\|v^{2,\varrho}\|_{5,Q_{T_0}}\leq c \|((v^1)_\varrho+(v^{2,\varrho})_\varrho)\otimes ((v^1)_\varrho+v^{2,\varrho})\|_{5,Q_{T_0}} $$
with an absolute constant $c$. So, we have
$$\|v^{2,\varrho}\|_{5,Q_{T_0}}\leq c(\|v^{2,\varrho}\|^2_{5,Q_{T_0}} + \|v^{2,\varrho}\|_{5,Q_{T_0}}\|v^{1}\|_{5,Q_{T_0}}+\|v^{1}\|^2_{5,Q_{T_0}})$$
If we assume that
\begin{equation}\label{smallness}\|v^{1}\|_{5,Q_{T_0}}\leq \frac 1{5c},\end{equation}
then it is not difficult to show that
$$\|v^{2,\varrho}\|_{5,Q_{T_0}}\leq\frac 12\|v^{1}\|_{5,Q_{T_0}}.$$
And the same bound is true  for the limit function. So,  $\|v^{2}\|_{5,Q_{T_0}}<\infty$. Condition (\ref{smallness}) gives an estimate on $T_0$. \end{proof}


The statement of Theorem  \ref{mainuniqueness}  follows immediately from Proposition \ref{uniqueness1}
and Corollary \ref{mild} reading that any weak $L_3$-solutions is a mild solution
on a short time interval.

\setcounter{equation}{0}
\section{Appendix}

The aim of this section is to give
an elementary proof of the existence of a mild solution to the Cauchy problem
(\ref{system})--(\ref{ic}). To this end, let us consider first the following Stokes problem
$$\partial_tw-\Delta w+\nabla r=-{\rm div}\,F,\qquad {\rm div}\,w=0$$
in $Q_\infty$ and
$$w(\cdot,0)=0$$
in $\mathbb R^3$.

Given $F\in C^\infty_0(Q_T;\mathbb R^{3\times 3})$, there exists a unique function $w$ such that $w\in C([0,T];L_3(\mathbb R^3))\cap L_5(Q_T)$ and $\sqrt{|w|}\nabla w\in L_2(Q_T)$ with the estimate
$$\|w\|_{3,\infty,Q_T}+\|w\|_{5,Q_T}\leq c\|F\|_{\frac 52,Q_T},$$
where $c>0$ is an absolute constant.
Moreover, $w$ can be expressed in the following way
$$w(x,t)=\int\limits^t_0\int\limits_{\mathbb R^3}K(x-y,t-\tau)F(y,\tau)dyd\tau.$$
So, we have the linear integral operator $\mathcal G:C^\infty_0(Q_T)\subset L_\frac 52(Q_T)\to L_\frac 52(Q_T)\cap C([0,T];L_3(\mathbb R^3))$. We denote by the same symbol the extension of this operator to the whole space $L_\frac 52(Q_T)$. We wish to show that
$$\mathcal GF(x,t)=\mathcal G_0F(x,t):=\int\limits^t_0\int\limits_{\mathbb R^3}K(x-y,t-\tau)F(y,\tau)dyd\tau$$
for $F\in L_\frac 52(\mathbb R^3)$. To this end, we first shall show that $\mathcal G_0: L_\frac 52(\mathbb R^3)\to
 L_4(\mathbb R^3)$ is bounded. We know, see for example \cite{KNSS2009}, that
 $$|K(x,t)|\leq K_0(x,t):=\frac c{(|x|^2+t)^2}.$$
Let
$$g(x,t):=\int\limits^t_0\int\limits_{\mathbb R^3}K_0(x-y,t-\tau)|F(y,\tau)|dyd\tau$$
and $$s=\frac {20}{17}.$$
Then we have
$$|K_0(x-y,t-\tau)F(y,\tau)|=|K_0(x-y,t-\tau)|^\frac {5}{17}|F(y,\tau)|^\frac 58\times|F(y,\tau)|^\frac 38\times $$$$\times|K_0(x-y,t-\tau)|^\frac{12}{17}.$$
By H\"older inequality, we find
$$g(x,t)\leq c\Big(\int\limits^t_0\int\limits_{\mathbb R^3}\frac{1}{(|x-y|^2+t-\tau)^\frac {40}{17}}|F(y,\tau)|^\frac 52dyd\tau\Big)^\frac 14\times$$
$$\times \Big(\int\limits^t_0\int\limits_{\mathbb R^3}|F(y,\tau)|^\frac 52dyd\tau\Big)^{\frac 25-\frac 14}\times
\Big(\int\limits^t_0\int\limits_{\mathbb R^3}\frac{1}{(|x-y|^2+t-\tau)^\frac {40}{17}}dyd\tau\Big)^\frac 35.$$
The last factor can be evaluated as follows:
$$\int\limits^t_0\int\limits_{\mathbb R^3}\frac{1}{(|x-y|^2+t-\tau)^\frac {40}{17}}dyd\tau=c\int\limits^t_0(t-\tau)^{-\frac {29}{34}}d\tau\int\limits^\infty_0\frac {r^2dr}{(r^2+1)^\frac {40}{17}}\leq C(T).$$
So, we have
$$\|g(\cdot,t)\|^4_{4,\mathbb R^3} \leq C(T)\|F\|_{\frac 52,Q_T}^\frac 32  \int\limits_{\mathbb R^3}dx\int\limits^t_0\int\limits_{\mathbb R^3}\frac{1}{(|x-y|^2+t-\tau)^\frac {40}{17}}|F(y,\tau)|^\frac 52dyd\tau = $$
$$=C(T)\|F\|_{\frac 52,Q_T}^\frac 32
\int\limits^t_0(t-\tau)^{-\frac {29}{34}}\int\limits_{\mathbb R^3}|F(y,\tau)|^\frac 52dyd\tau\leq $$
$$\leq C(T)\|F\|_{\frac 52,Q_T}^\frac 32
\int\limits^T_0|t-\tau|^{-\frac {29}{34}}\int\limits_{\mathbb R^3}|F(y,\tau)|^\frac 52dyd\tau.$$
Hence,
$$\|g(\cdot,t)\|^4_{4,Q_T}\leq C(T)\|F\|_{\frac 52,Q_T}^\frac 32
\int\limits^T_0\int\limits_{\mathbb R^3}|F(y,\tau)|^\frac 52dyd\tau \int\limits^T_0|t-\tau|^{-\frac {29}{34}}dt\leq$$
$$ \leq C(T)\|F\|_{\frac 52,Q_T}^4.$$

Next, our arguments are as follows. One can find a sequence $F^{(m)}\in C^\infty_0(Q_T)$ such that $F^{(m)}\to F$ in $L_\frac 52(Q_T)$ as $m\to\infty$. Let $w^{(m)}$ be a solution to the above Stokes system. It is easy to check that
$$w^{(m)}(x,t)=\int\limits^t_0\int\limits_{\mathbb R^3}K(x-y,t-\tau)F^{(m)}(y,\tau)dyd\tau.$$
As we know, the following estimate is valid:
$$\|w^{(m)}-w^{(k)}\|_{3,\infty,Q_T}+\|w^{(m)}-w^{(k)}\|_{5,Q_T}\leq c\|F^{(m)}-F^{(k)}\|_{\frac 52,Q_T}.$$
Let $w\in C([0,T];L_3(\mathbb R^3))\cap L_5(Q_T)$ be a limit function. It is a unique solution to the Stokes problem with the limit function $F$. We need to show that
$w=\mathcal G_0F$. Indeed,
$$\|\mathcal G_0F -\mathcal G_0F^{(m)}\|_{4,Q_T}=\|\mathcal G_0F -w^{(m)}\|_{4,Q_T}\leq C(T)\|F-F^{(m)}\|_{\frac 52,Q_T}.$$
Simply by interpolation, we can state that $w\in L_4(Q_T)$ and
$\|w-w^{(m)}\|_{4,Q_T}\to 0$ and, hence,
$$w=\mathcal G(F).$$

Now, our further arguments are quite standard. We let
$$V(\cdot,t)=\Gamma (\cdot,t)\star v_0(\cdot), \qquad \kappa(T)=\|V\|_{5,Q_T} $$
and let
$$v^{(k+1)}=V+\mathcal G(v^{(k)}\otimes v^{(k)})$$
with $v^{(0)}=0$. Using previous estimates, we have
$$\|v^{(k+1)}-V\|_{3,\infty,Q_T}+\|v^{(k+1)}-V\|_{5,Q_T}
\leq c\|v^{(k)}\|^2_{5,Q_T}.$$
Thus
$$\|v^{(k+1)}\|_{5,Q_T}\leq \|V\|_{5,Q_T}+c\|v^{(k)}\|^2_{5,Q_T}. $$ The  objective, now, is to show
$$\|v^{(k+1)}\|_{5,Q_T}\leq 2\kappa(T).$$
Arguing by induction, we arrive at the estimate
$$\|v^{(k+1)}\|_{5,Q_T}\leq c4\kappa^2(T)+\kappa(T)=\kappa(T)(4c\kappa(T)+1).$$
Let us impose the additional assumption
$$\kappa(T)\leq \frac 1 {16c}.$$
Later on, we shall show that it is possible. Now, assume that the above condition holds. We have $$v^{(k+1)}-v^{(k)}=\mathcal G(v^{(k)}\otimes v^{(k)}-v^{(k-1)}\otimes v^{(k-1)})$$
and thus
$$\|v^{(k+1)}-v^{(k)}\|_{3,\infty,Q_T}+\|v^{(k+1)}-v^{(k)}\|_{5,Q_T}\leq$$
$$\leq 2c\|v^{(k)}-v^{(k-1)}\|_{5,Q_T}(\|v^{(k)}\|_{5,Q_T}+\|v^{(k-1)}\|_{5,Q_T})\leq $$
$$\leq 8c\kappa(T)\|v^{(k)}-v^{(k-1)}\|_{5,Q_T}\leq \frac 12\|v^{(k)}-v^{(k-1)}\|_{5,Q_T}\leq \frac 14\|v^{(k-1)}-v^{(k-2)}\|_{5,Q_T}\leq$$ $$\leq \frac 1{2^{k-1}}\|v^{(1)}\|_{5,Q_T}= \frac 1{2^{k-1}}\kappa(T).$$
Therefore,
$$
\|v^{(k)}-v^{(m)}\|_{5,Q_T}\leq\sum\limits_{i=m}^{k-1}\|v^{(i+1)}-v^{(i)}\|_{5,Q_T}\leq $$
$$\leq \sum\limits_{i=m}^{k-1}\frac 1{2^{i}}\kappa(T).$$
So, $v^{(m)}\to v$ in $L_5(Q_T)$. Then
$$\|v^{(k)}-v^{(m)}\|_{3,\infty,Q_T}\leq 2c\|v^{(k)}-v^{(m)}\|_{5,Q_T}(\|v^{(k)}\|_{5,Q_T}+\|v^{(m)}\|_{5,Q_T})\leq $$$$\leq 8c\|v^{(k)}-v^{(m)}\|_{5,Q_T}\kappa(T).$$
This means that $v^{(m)}\to v$ in $C([0,T];L_3(\mathbb R^3))$. So, the existence of mild solution has been proven.
Uniqueness follows easily form the same arguments as above.
See the additional assumption when proving strong convergence of the whole sequence.

Now, going back to the assumption on the smallness of $\kappa(T)$, we have
$$\kappa(T)=I^1_\varrho+I^2_\varrho,$$
where
$$I^1_\varrho=\|(V)_\varrho\|_{5,Q_T}, \quad I^2_\varrho=
\|V-V_\varrho\|_{5,Q_T}, \quad V_\varrho(\cdot,t)=\Gamma(\cdot,t)\star (v_0)_\varrho(\cdot),$$
and $(v_0)_\varrho$ is a standard mollification of $v_0$.
With $I^2_\varrho$, we proceed as follows
$$I^2_\varrho\leq c\|v_0-(v_0)\varrho\|_{3,\mathbb R^3}\leq \frac 1{32c}$$
for some fixed $\varrho>0$.  Next,
$$\|(V)_\varrho(\cdot,t)\|_{5,\mathbb R^3}\leq ct^{-\frac 1r}\|(v_0)_\varrho\|_{4,\mathbb R^3},$$
where
$$\frac 1r=
\frac{3}{2}\Big(\frac 14-\frac 15\Big)=\frac{3}{40}.$$
Hence,
$$I^1_\varrho=\|(V)_\varrho\|_{5,Q_T}\leq c T^\frac {1}{8}\|(v_0)_\varrho\|_{4,\mathbb R^3}.$$
The right hand side of the latter inequality can be made small for a given $\varrho$ at the expense of $T$.

Now, I wish to show that the constructed above mild solution is in fact a weak $L_3$-solution in $Q_T$.
To this end we need to show that $w:=v-v^1\in L_{2,\infty}(Q_T)\cap W^{1,0}_2(Q_T)$ and satisfy the energy inequality.
We start with local energy inequality
$$I:=\frac 12\int\limits_{\mathbb R^3}\varphi^2(x) |w(x,t)|^2dx+\int\limits^t_0\int\limits_{\mathbb R^3}\varphi^2|\nabla w|^2dxds\leq $$
$$\leq \int\limits^t_0\int\limits_{\mathbb R^3}\Big(\frac 12 |w|^2\Delta \varphi^2+\frac 12w\cdot\nabla \varphi^2(|w|^2+2r)+$$$$+v^1\otimes w:\nabla w\varphi^2+v^1\otimes w:w\otimes \nabla\varphi^2+$$
$$+w\otimes v^1:\nabla w\varphi^2+w\otimes v^1:w\otimes \nabla\varphi^2+$$$$+v^1\otimes v^1:\nabla w\varphi^2+v^1\otimes v^1:w\otimes \nabla\varphi^2\Big)dxds$$
for any $\varphi\in C^\infty_0(\mathbb R^3)$. Here, the pressure $r$ is defined by the equation
$$\partial_tw-\Delta w+\nabla r=-{\rm div}\,v\otimes v$$ and we know
$$\|r\|_{\frac 52,Q_T}+\|r\|_{\frac 32,\infty,Q_T}<\infty.$$

Assuming that $0\leq \varphi\leq 1$, $\varphi(x)=1$, if $|x|<R$, $\varphi(x)=0$ if $|x|>2R$, and $|\nabla\varphi|\leq c/R$,  we  find
$$I\leq c\frac 1{R^2}\|w\|^2_{3,Q_T}R+c\frac 1R\|w\|^3_{3,Q_T}+c\frac 1R\|w\|_{3,Q_T}\|r\|_{\frac 32,Q_T}+$$$$+c(\|v^1\|_{4,Q_T}\|w\|_{4,Q_T}+\|v^1\|^2_{4,Q_T})I^\frac 12+$$$$+c\frac 1R(\|v^1\|_{3,Q_T}\|w\|^2_{3,Q_T}+\|v^1\|^2_{3,Q_T}\|w\|_{3,Q_T})$$
From the latter bound, we can easily get all the statements.
\begin{cor}\label{mild} Let be $v$ be a weak $L_3$-solution. There exists $T>0$ such that $v$ is a mild solution in $Q_T$.
\end{cor}
Indeed, we know that there exist a mild solution $u$ in $Q_T$ for some $T>0$ depending on $v_0$. By the previous observation, it is a weak $L_3$-solution in $Q_T$. By the uniqueness theorem, $v=u$.

\end{document}